\documentclass[letterpaper,11pt]{article}
\usepackage{graphicx}
\usepackage{amsmath, mathrsfs}
\usepackage{amssymb}
\usepackage{amsthm}
\usepackage{cite}
\usepackage{xcolor}
\usepackage{soul}
\usepackage[all]{xy}
\usepackage{blkarray} 
\usepackage{hyperref}
\hypersetup{
    colorlinks=true,
    linkcolor=blue,
    filecolor=magenta,
    urlcolor=cyan,
    citecolor=cyan,
    pdftitle={Overleaf Example},
    pdfpagemode=FullScreen,
}

\newcommand{\bs}[1]{{\boldsymbol{#1}}}
\newcommand{\gr}[1]{( #1 )}
\newcommand{\Gr}[1]{\big( #1 \big)}
\newcommand{\GR}[1]{\Big( #1 \Big)}

\newcommand{\AutoGroup}[1]{\left( #1 \right)}
\newcommand{\abs}[1]{\vert #1 \vert}
\newcommand{\Abs}[1]{\big\vert #1 \big\vert}

\newcommand{\set}[1]{\{ #1 \}}
\newcommand{\Set}[1]{\big\{ #1 \big\}}

\newcommand{\vect}[1]{[ #1 ]}
\newcommand{\Vect}[1]{\big[ #1 \big]}

\newcommand{\Real}{\mathbb R}
\newcommand{\C}{\mathbb C}
\newcommand{\Int}{\mathbb Z}
\newcommand{\e}{\mathrm{e}} 
\newcommand{\F}{\mathcal{F}}
\newcommand{\conv}{\ast}
\newcommand{\Cbasis}{\bs{P}} 

\newcommand{\dimin}[2]{{ #1 }^{\gr{#2}}}
\newcommand{\bcode}[1]{\texttt{#1}}
\newcommand{\norm}[1]{\Vert #1 \Vert}
\newcommand{\Norm}[1]{\big\Vert #1 \big\Vert}
\newcommand{\ip}[2]{\langle #1, #2 \rangle}

\newcommand{\dotprod}[2]{\langle\!\langle #1, #2 \rangle\!\rangle}
\newcommand{\Dotprod}[2]{\big\langle\!\:\!\!\big\langle #1, #2 \big\rangle\!\:\!\!\big\rangle}

\newcommand{\alertblue}[1]{\textcolor{blue}{#1}}

\newcommand{\hatbsf}{{\kern+0.25em}\hat{{\kern-0.25em}\bs{f}}}
\newcommand{\imag}{\mathfrak{i}}

\newcommand{\ith}{$i$th}
\newcommand{\jth}{$j$th}
\newcommand{\kth}{$k$th}

\newcommand{\ellth}{$\ell$th}
\newcommand{\mth}{$m$th}

\newcommand{\qth}{$q$th}

\newcommand{\ijth}{$(i,j)$th}
\newcommand{\klth}{$(k,l)$th}

\newtheorem{thm}{Theorem}[section]

\theoremstyle{definition} \newtheorem{remark}{Remark}
\newtheorem{definition}{Definition}[section]

\begin{document}

\title{Statistics of a multi-factor function\\ from its Fourier transform}

\author{Matthew A. Herman\thanks{Corresponding author. The final parts of this paper were completed while MAH was at the Ogbunu Lab, Yale University.}\qquad \qquad\quad Stephen Doro \\ \;\;\,\normalsize\texttt{matthew.herman@yale.edu} \qquad\; \texttt{sd15@columbia.edu} \;}

\date{\today}

\maketitle

\begin{abstract}
For a phenomenon~$\bs{f}$ that is a function of~$n$ factors, defined on a finite abelian group $G$, we derive its population statistics solely from its Fourier transform~$\hatbsf$. Our main result is an \emph{$m$-Coefficient/Index Annihilation Theorem}: the \mth\ moment of~$\bs{f}$ becomes a series of terms, each with precisely $m$ Fourier coefficients --- and surprisingly, the coefficient \emph{indices} in each term sum to zero under group addition. This condition acts like a filter, limiting which terms appear in the Fourier domain, and can reveal deeper relationships between the variables driving~$\bs{f}$. These techniques can also be used as an analytical/design tool, or as a feasibility constraint in search algorithms. For functions defined on $\Int_2^n$, we show how the skew, kurtosis, etc.~of a binomial distribution can be derived from the Fourier domain. Several other examples are presented.
\end{abstract}


\section{Introduction}
The classic ``bell-shaped curve'' is a profound and central tool in statistics, but it is widely applied to data which violate, even if slightly, the assumptions upon which it is founded. This paper is an attempt to understand data distributions when nonlinearities are present but small, sparse, and low-degree.

The Gaussian distribution amalgamates the net effect of multiple factors, when each is small and combine linearly. Its use was familiar in physics as a way to composite error terms when Francis Galton introduced the Gaussian distribution into biology, using it as the very trait variability which characterizes a population and drives evolution.
\begin{quote}
\emph{I know of scarcely anything so apt to impress the imagination as the wonderful form of cosmic order expressed by the ``Law of Frequency of Error.''\ldots\
The huger the mob, and the greater the apparent anarchy, the more perfect is its sway. It is the supreme law of Unreason.}~\cite[p.66]{galton1889natural}
\end{quote}

But Galton's encomium to the Gaussian distribution does not survive contact with real data, when multiple contributing factors may combine in nonlinear ways, e.g., via epistatic interactions in genetics and wave coupling in fluid dynamics, creating non-Gaussian distributions. This has led to a curious mismatch between data sets in diverse fields~\cite{taleb2023statistical}, which are often non-Gaussian, and analytical tools which assume linear combination of effects --- biology and fluid dynamics are just two of many striking examples of this. The discrepancy between linear analytic tools and nonlinear processes has led to intensified interest in developing higher-order statistics~\cite{mendel1991tutorialHigherOrderStats, swami1997bibHOS}.

Fourier analysis provides a promising approach. There is both theoretical and holistic value in understanding how ``both sides of the Fourier coin'' contribute to a function's statistics. The amount of nonlinearity may be limited by assumptions that interactions are sparse and of low degree. Fourier techniques generate a plethora of distribution curves with variable skewness, kurtosis, and higher moments by titrating the phase and magnitude of nonlinear/interaction terms. As such, we are able to achieve a gentle transition from linear approximations to more realistic nonlinear systems.

\subsection{Overview of paper}
This paper grew out of work in~\cite{doro2021fourierTransQuantTraitCS}, where we were interested in how information about the roughness of a trait's genetic landscape is contained in its Fourier transform. This naturally led us to ponder how the Fourier transform can reveal the statistics of a general function under the influence of multiple agents. For a function, or process~$\bs{f}$ defined on a finite abelian group, we ask the question
$$
\text{\emph{How do the population moments of~$\bs{f}$ depend on its Fourier transform~$\hatbsf$?}}
$$
Our main tool to answer this is~\eqref{Eq:mu_m^a=g*gtg*g}. This generates Theorem~\ref{Thm:mthMoment_a}, which has two salient components: (1) a closed-form expression for any moment of~$\bs{f}$ in terms of~$\hatbsf$, and (2) an index annihilation property that limits how the Fourier coefficients can combine.

We only address the situation where each factor is of finite length, but incorporating infinite-length Fourier transforms is also possible for some or all of the factors, at the expense of nontrivial functional analysis. We also assume access to all observations of a given process (thereby deeming it ``deterministic''); however, stochastic processes can be accommodated by employing an expectation operator, assuming a probability density or mass function is known or assumed.

\subsubsection{Applications}
On the practical side, we envision these methods being applied to scenarios including
\begin{itemize}
\item \emph{Calculating} the statistics of a function from its Fourier transform, e.g., if the full set of outcomes of~$\bs{f}$ is incomplete/corrupted but we are able to somehow obtain or assume~$\hatbsf$, perhaps using compressive sensing~\cite{CanRomTao_Noise}, a sparse Fourier transform~\cite{gross2022sparseFourierManyVariables}, or some other machine learning technique. This is probably most useful for smaller problems or if~$\hatbsf$ is sparse.

\item \emph{Analyzing} or \emph{designing} a function's behavior in its Fourier domain before synthesizing it (whether mathematically or in a possibly costly physical form). In line with this, there are many scenarios that are analogous to working in the Fourier domain, such as graphs/hypergraphs and association schemes~\cite{delsarte1973algebraic}.

\item As a \emph{feasibility constraint} in search/optimization algorithms. Possible applications include phase recovery and blind deconvolution, e.g., X-ray crystallography and channel equalization, where only the magnitude of the coefficients of~$\hatbsf$ is known~\cite{wolf2011historyXrayPhaseProblem}.
\end{itemize}
Although not necessary, if the support of~$\hatbsf$ is sparse or somehow structured (e.g., low-degree or low-frequency concentration or representing some underlying geometry), then operating in the Fourier domain is even more justified.

\subsubsection{Methodology and paper layout}
Our $n$-dimensional model follows the vectorized setup in~\cite{good1958interaction}; the mixing property of the Kronecker product organizes all combinations into a mathematically coherent structure that naturally paves the way to identify interactions (see Remark~\ref{Rem:Mixing_property}). We adopt a linear algebra approach using elementary Fourier theorems and circulant matrices to effect autoconvolutions. Specifically, analytical expressions for the \mth\ moment of~$\bs{f}$ can be found from the \mth\ power of the circulant matrix associated with~$\hatbsf$. Although explicit convolutions are required to find these representations, the computational cost should not be too severe, especially if~$\hatbsf$ is sparse.

The paper is organized as follows. Section~\ref{Sec:MathSetup} reviews the necessary statistics and mathematics and sets up our notation. Our main results are presented in Section~\ref{Sec:Moments_in_Fourier_domain}, which are specialized in Section~\ref{Sec:SpecialCase_Z2n} for the case of $G=\Int_2^n$. Section~\ref{Sec:Applications} discusses some example applications, and Section~\ref{Sec:Conclusion} contains concluding remarks. Appendix~\ref{App:TimeSeries_mthOrderMoment_nDim} shows how the \mth-order moment in a time series can be extended to $n$ dimensions.

\subsubsection{Connection with previous work}
During the final stages of completing this paper, we became aware of a large corpus of work in the signal processing and time series communities (including plasma physics, oceanography, fluid dynamics, general wave theory, and nonlinear optics) that use Fourier analysis to analyze higher-order moments, cumulants,\footnote{Higher-order moments and cumulants are used to identify functions that depart from ``Gaussianity.'' However, our methods are also relevant to Gaussian distributions when used as a \emph{constraint}; see Remark~\ref{Rem:Normal_constraint} in Section~\ref{Sec:FeasibilityConstraint_DFT_example}.} and their associated polyspectra, e.g., see~\cite{mendel1991tutorialHigherOrderStats, nikias1993higher, brillinger1983timeseriesfreqdom, priestley1988nonlinearnonstationarytimeseries, swami1997bibHOS}. However, as far as we can tell, in the discrete case, these models all seem to be defined on~$\Int^n$ or~$\Int_N^n$, for dimension $n=1,2,3$.

Separately, the field of \emph{experimental design} uses general $n$-dimensional transforms (usually not Fourier) on full factorial models to identify interactions between~$n$ factors affecting an observed phenomenon --- each factor can be thought of as existing along a unique axis, and any interactions (or ``entanglements'') are between axes. This coincides with the $n$-dimensional Fourier transform when the factors are strictly binary, realized by the Sylvester-Hadamard matrix of order~$2^n$. Good~\cite{good1958interaction} summarized the efficient mathematical methods used on full factorial models in conjunction with Kronecker products and generalized it to Fourier analysis.\footnote{Good's paper~\cite{good1958interaction} was a major motivation to Cooley and Tukey's seminal FFT algorithm\cite{cooleytukey1965fft}.}

The Kronecker product facilitates juggling multidimensional indices\footnote{Or the ``debauch of indices'' in Spivak's memorable description of the tensor calculus~\cite[Ch.~5]{spivak1979comprehensive}.} in the natural endeavor of extending Fourier techniques to the statistics of time series beyond one dimension. Moreover, its use in multi-factor circulant matrices~\eqref{Def:PermMatrix_KG_basis},~\eqref{Def:CircMatrix_KG_basis} is the mechanism that directly leads to our derivation of the \mth\ moment, $\dimin{\mu}{a}_m\gr{\bs{f}}$. That is, autoconvolutions via powers of circulant matrices~\eqref{Eq:C_*^mf=C_f^m} naturally do the heavy lifting for us --- they implicitly carry out the exhaustive combinatorial work, determining which~$m$ Fourier coefficients are permitted to appear in terms satisfying the annihilation constraint (i.e., $j_1 \oplus \cdots \oplus j_m = 0$) in~\eqref{ThmEq:mthMoment_a}.

As far as we can tell, this perspective/technique does not appear in the literature.
In line with this, our derivation of the annihilation condition seems to be unique (see the end of proof for Theorem~\ref{Thm:mthMoment_a}); typically it is obtained via the approach in Appendix~\ref{App:TimeSeries_mthOrderMoment_nDim}, or by invoking Fourier-Stieltjes representation theory~\cite[Ch.~2.3.2]{nikias1993higher},~\cite[Ch.~3.2]{priestley1988nonlinearnonstationarytimeseries} or by perturbing a Hamiltonian system~\cite[Ch.~1]{morbidelli2002modernCelestialMechanics}.

\section{Mathematical setup} \label{Sec:MathSetup}
\subsection{Statistics background}
In this section, we briefly review the basic statistical definitions and concepts necessary for later in the paper. Assume $\bs{f}: G\rightarrow\C$ for some finite abelian group~$G$ with $\bs{f} = \gr{f_0,f_1,\ldots,f_{\abs{G}-1}}^\top$, where~$\gr{\cdot}^\top$ denotes the (non-conjugate) transpose. In practice, group~$G$ represents our sample space.

\subsubsection{General, raw, and central populations moments} \label{Sec:Gen_raw_central_momennts}
The \emph{\mth\ general moment of~$\bs{f}$ centered at point~$a\in\C$} is defined for integer $m\geq0$ as
\begin{equation} \label{Def:mthMoment_a}
\dimin{\mu}{a}_m\gr{\bs{f}} \,:=\; \frac{1}{\abs{G}} \sum_{i\in G} \gr{f_i-a}^m.
\end{equation}
The \emph{\mth\ raw moment of~$\bs{f}$} (denoted with an accent mark) coincides with the origin, $a=0$, and the \emph{\mth\ central moment of~$\bs{f}$} (with no superscript) coincides with the population mean, $a=\mu$:
\begin{equation} \label{Def:mthRawCentralMoments}
\mu'_m\gr{\bs{f}} \,:=\; \frac{1}{\abs{G}} \sum_{i\in G} f_i^m, \quad\qquad
\mu_m\gr{\bs{f}} \,:=\; \frac{1}{\abs{G}} \sum_{i\in G} \gr{f_i-\mu}^m.
\end{equation}

The base cases, $m=0,1$, for the general moment of~$\bs{f}$ are straightforward. It should be clear that
\begin{equation} \label{Eq:mthMoment_base_cases}
\dimin{\mu}{a}_0\gr{\bs{f}} = 1 \qquad\text{and}\qquad \dimin{\mu}{a}_1\gr{\bs{f}} = \mu-a
\end{equation}
for any $a\in\C$.
For $m=1$, this trivially reduces to $\mu'_1\gr{\bs{f}} = \mu$ when $a=0$, and $\mu_1\gr{\bs{f}} = 0$ when $a=\mu$.

The moments defined above can take on complex values if~$\bs{f}$ is a complex-valued function. The \emph{variance}, however, is always real and nonnegative:
\begin{equation} \label{Def:Variance}
\sigma^2\gr{\bs{f}} \,:=\; \frac{1}{\abs{G}} \sum_{i\in G} \abs{f_i-\mu}^2.
\end{equation}
Of course, $\mu_2 = \sigma^2$ when~$\bs{f}$ is real.

\subsubsection{Standardized moments}
Consider a histogram of all values of a given function~$\bs{f}$. There are many ways to describe the shape of this distribution. One of these measures is \emph{standardized moments}, which are just the central moments normalized by appropriate powers of the variance: $\mu_m/\sigma^m$ (with $\sigma\neq0$).

In brief, for a given distribution:
the \emph{skewness,} $\gamma := \mu_3/\sigma^3$, assesses its symmetry;
the \emph{kurtosis,} $\kappa := \mu_4/\sigma^4$, measures the fatness of its tails;
the \emph{hyperskewness,} $\gamma_5 := \mu_5/\sigma^5$, captures any extreme asymmetries;
and the \emph{hyperkurtosis,} $\kappa_6 := \mu_6/\sigma^6$, embodies how extreme the kurtosis is. Standardized moments for $m>6$ can be calculated if needed.

\subsection{The domain}
Let us first consider a $1$-dimensional setting, i.e., a single variable or factor with $N$ equally spaced values, which is modeled by $N$ equally spaced points on the unit circle. This is the finite cyclic group $\Int_N := \Int/N\Int = \set{0,1,\ldots,N-1}$, the ring of integers modulo~$N$, with cardinality~$\abs{\Int_N}=N$.

Now consider~$n$ factors, where the \ellth\ factor has $N_\ell$ equally spaced values. This $n$-dimensional setup is modeled\footnote{More general models are possible, such as the direct product of~$K$ objects, $G = G_1 \times G_2 \times \cdots \times G_K$, where the \kth\ object,~$G_k$, is of the form~\eqref{Eq:FiniteAbelianGroup_G} with $n_k$ dimensions. For example, a system of $K$ different waves in $3$-dimensional space would have $n_k=3$ for each $k=1,2,\ldots,K$.} by the direct product of~$n$ finite cyclic groups
\begin{equation} \label{Eq:FiniteAbelianGroup_G}
G \,=\, \Int_{N_1} \times \Int_{N_2} \times \cdots \times \Int_{N_n}
\end{equation}
which is a finite abelian group with cardinality $\abs{G} = N_1 N_2 \cdots N_n$.

The elements of~$G$ are \emph{$n$-tuples}
\begin{equation} \label{Eq:i_n-tuple}
i := \vect{\bcode{i}_1, \bcode{i}_2, \ldots, \bcode{i}_n} \in G
\end{equation}
(always in ``TrueType'' font and enclosed with square brackets), where $\bcode{i}_\ell\in\Int_{N_\ell}$ for $\ell=1,2,\ldots,n$. Defining~$j\in G$ similarly, we can add~$i$ and~$j$ in a natural way, modulo~$N_\ell$ in the \ellth\ direction. Thus, the \emph{group operation of addition in~$G$} is simply
$$
i \oplus j \,:=\, \vect{\bcode{i}_1+\bcode{j}_1 \!\!\!\!\pmod{N_1},\, \bcode{i}_2+\bcode{j}_2 \!\!\!\!\pmod{N_2},\, \ldots\, , \, \bcode{i}_n+\bcode{j}_n \!\!\!\!\pmod{N_n}}.
$$
In the $1$-dimensional case with $G = \Int_N$, $i \oplus j$ reduces to regular addition modulo~$N$.
It trivially follows that \emph{group subtraction} is
$i \ominus j \,:=\, i \oplus \gr{\ominus j}$, where the \emph{additive inverse} of~$j\in G$ is
$\ominus j \,:=\, \Vect{\bcode{N}_1-\bcode{j}_1, \bcode{N}_2-\bcode{j}_2, \ldots, \bcode{N}_n-\bcode{j}_n}$.

Sometimes, it is easier to represent an $n$-tuple as a decimal integer. This can be achieved by the mapping $i = \sum_{\ell=1}^n \bcode{i}_\ell \cdot \AutoGroup{\prod_{k=\ell+1}^n N_k}$, but other mappings exist too. For a given mapping, each $n$-tuple~$i\in G$ has a unique decimal equivalent. When $N_\ell=N$ for all $\ell=1,2,\ldots,n$, then this is just the standard conversion from base-$N$ to decimal. For example, if each $N_\ell=2$, then the elements of $G=\Int_2^n$ are $n$-bit codes and this is the conversion from binary to decimal, $i = \sum_{\ell=1}^n \bcode{i}_\ell \cdot 2^{\ell-1}$.
The lexicographical order of the $n$-tuples coincides with the sequential order of their decimal equivalents, so unless necessary, we make no distinction between them; the context should make it clear.

\subsection{The space of functions} \label{Sec:Space_of_functions}
For finite abelian group $G$, let $L\gr{G}=\set{\bs{f}:G\rightarrow\C}$ be the set of complex-valued functions defined on $G$ with column vector\footnote{The vectorization or ``flattening'' of function~$\bs{f}$ is based on the lexicographic ordering of group $G$.} $\bs{f} = \gr{f_0,f_1,\ldots,f_{\abs{G}-1}}^\top$. The natural \emph{inner product} for such a space is $\ip{\bs{f}}{\bs{g}} := \sum_{i\in G} f_i \overline{g_i}$, for any $\bs{f},\bs{g}\in L\gr{G}$, where~$\overline{\:\cdot\:}$ denotes complex conjugation, which induces the \emph{Euclidean norm}, $\norm{\bs{f}}=\ip{\bs{f}}{\bs{f}}^{{1}/{2}}$. Similar to the inner product is the simple \emph{dot product} (with double angle brackets),
\begin{equation} \label{Def:dotProduct}
\qquad\dotprod{\bs{f}}{\bs{g}} \,:=\; \sum_{i\in G} f_i \:\! g_i \,=\, \ip{\bs{f}}{\overline{\bs{g}}}, \qquad \text{for any } \bs{f},\bs{g}\in L\gr{G}.
\end{equation}
We point out that $\ip{\bs{f}}{\bs{g}} = \overline{\ip{\bs{g}}{\bs{f}}}$, however $\dotprod{\bs{f}}{\bs{g}} = \dotprod{\bs{g}}{\bs{f}}$.

Recall the operation of group subtraction. For $j\in G$, let~$\mathcal{S}_j$ be the \emph{$n$-dimensional circulant shift operator} defined by
\begin{equation} \label{Def:ShiftOperator}
\mathcal{S}_j\bs{f} \,:=\, \gr{f_{0\ominus j},f_{1\ominus j},\ldots,f_{\gr{\abs{G}-1}\ominus j}}^\top
\end{equation}
where in the \ith\ position, $f_{i\ominus j} = f_{\vect{\bcode{i}_1-\bcode{j}_1 \!\!\pmod{N_1},\, \bcode{i}_2-\bcode{j}_2 \!\!\pmod{N_2},\, \ldots\, , \, \bcode{i}_n-\bcode{j}_n \!\!\pmod{N_n}}}$. When $n=1$ this reduces to the familiar circulant shift operator on $\Int_N$. Next, we overload the~$\ominus$ symbol to define~$\bs{f}_{\!\ominus}$ as the coefficients of~$\bs{f}$ listed in \emph{negative sequential order}:
\begin{equation} \label{Def:fnc_f_ominus}
\bs{f}_{\!\ominus} \,:=\, \gr{f_0,f_{\ominus1},\ldots,f_{\ominus\gr{\abs{G}-1}}}^\top
\end{equation}
where in the \jth\ position, $f_{\ominus j} = f_{\vect{\bcode{N}_1-\bcode{j}_1, \bcode{N}_2-\bcode{j}_2, \ldots, \bcode{N}_n-\bcode{j}_n}}$.

For $j\in G$, let $\bs{e}_j\in L\gr{G}$ be the \emph{standard basis element} with a `$1$' in the \jth\ position and `$0$' elsewhere. Then $\bs{I} = \bs{I}_G := \gr{\bs{e}_0,\bs{e}_1,\bs{e}_2,\ldots,\bs{e}_{\abs{G}-1}}$ is the canonical $\abs{G}\text{-by-}\abs{G}$ identity matrix.

\subsection{Some algebra}
For vectors $\bs{f},\bs{g}\in L\gr{G}$, their \emph{element-wise product} has the natural form
\begin{equation} \label{Def:fdotg}
\bs{f}\cdot\bs{g} \,:=\, \gr{f_0\:\!g_0,f_1\:\!g_1,\ldots,f_{\abs{G}-1}\:\!g_{\abs{G}-1}}^\top.
\end{equation}
For any $m\geq0$, we use the notation $\bs{f}^m$ to represent the \emph{\mth\ power of the individual elements of a vector}, $\bs{f}^m := \gr{f^m_0,\,f^m_1,\,\ldots\,,\,f^m_{\abs{G}-1}}^\top$, where~$\bs{f}^0 := \gr{1,1,\ldots,1}$.

Next, let $\bs{f} \conv \bs{g}$ represent the \emph{$n$-dimensional circular convolution} of vectors $\bs{f},\bs{g}\in L\gr{G}$, whose \ith\ element is defined as
\begin{equation} \label{Def:Conv_finiteAbelianGroupG}
\qquad\qquad \gr{\bs{f} \conv \bs{g}}_i \,:=\, \sum_{j\in G} f_{i \ominus j} \:\! g_j, \qquad \text{for } i\in G
\end{equation}
which has the typical form of $1$-dimensional convolution defined on~$\Int_N$. It is straightforward to show that convolution is commutative and associative in $n$ dimensions.

For any $m\text{-by-}n$ matrix $\bs{A}=\Gr{a_{i,j}}_{i,j}$ and $p\text{-by-}q$ matrix~$\bs{B}=\Gr{b_{k,l}}_{k,l}$, their \emph{Kronecker (or tensor) product} is defined as the $mp\text{-by-}nq$ matrix
\begin{equation} \label{Def:KroneckerProduct}
\bs{A}\otimes\bs{B} \,:=\;
\GR{a_{i,j}\bs{B}}_{i,j}
\end{equation}
where $i\in\Int_m$, $j\in\Int_n$, $k\in\Int_p$, $l\in\Int_q$.

\begin{remark}[Combinatorial mixing property] \label{Rem:Mixing_property}
The Kronecker product possesses a \emph{mixing property} such that every combination of the elements of~$\bs{A}$ and~$\bs{B}$ occurs exactly once in $\bs{A}\otimes\bs{B}$. Specifically, the unique mixture $a_{i,j}b_{k,l}$ is the $\gr{ip+k,jq+l}$th element of $\bs{A}\otimes\bs{B}$. This engenders a natural framework to identify interactions/coupling between different entities.
\end{remark}

\subsection{The DFT}
The discrete Fourier transform (DFT) is arguably the most often implemented type of Fourier transform.\footnote{Otherwise, Strang's claim in~\cite{strang1994wavelets} that the DFT's fast version, the FFT (e.g.,~\cite{cooleytukey1965fft}), is ``the most important numerical algorithm of our lifetime'' would be irrelevant.} The main actor in the DFT is the \emph{DFT matrix}. The most common form of the DFT is when it operates on $1$-dimensional functions, i.e., those defined on~$\Int_N$~\cite[Ch.~2]{terras1999FourierAnalysisFiniteGroups}. In this case the DFT matrix is just $\bs{U}_{\Int_N} := \Gr{\omega_N^{\ominus ij}}_{i,j\in\Int_N}$,
where $\omega_N := \e^{2\pi\imag/N}$ and $\imag := \sqrt{-1}$. The rows/columns of~$\bs{U}_{\Int_N}$ are orthogonal and have a Euclidean length of~$N$, so $\bs{U}_{\Int_N}^*\bs{U}_{\Int_N} = \bs{U}_{\Int_N}\bs{U}_{\Int_N}^* = N\bs{I}$,
where $\gr{\cdot}^*$ is the Hermitian transpose.

It is straightforward to extend the DFT matrix to~$n$ dimensions, each of possibly different size~\cite[\S11]{good1958interaction},~\cite[Ch.~10]{terras1999FourierAnalysisFiniteGroups}. Let~$G$ be the direct product of~$n$ finite cyclic groups~\eqref{Eq:FiniteAbelianGroup_G}. The multidimensional DFT matrix defined on this group is just the Kronecker product~\eqref{Def:KroneckerProduct} of the DFT matrices associated with each~$\Int_{N_\ell}$:
\begin{equation} \label{Eq:Multidim_U_G}
\bs{U}_{\:\!\!G} \,=\, \bs{U}_{\Int_{N_1}} \otimes \bs{U}_{\Int_{N_2}} \otimes \cdots \otimes  \bs{U}_{\Int_{N_n}}.
\end{equation}
For $i,j\in G$, its \ijth\ element is the unique mixture of powers of the~$n$ roots of unity (see Remark~\ref{Rem:Mixing_property}):
\begin{equation} \label{Eq:U_G(i,j)}
\bs{U}_{\:\!\!G}\gr{i,j} = \prod_{\ell=1}^n \omega_{N_\ell}^{\ominus\bcode{i}_\ell\bcode{j}_\ell}.
\end{equation}
Due to other properties of the Kronecker product~\cite[\S4.2]{horn1994topics}, the DFT matrix obeys $\bs{U}_{\:\!\!G}^* \bs{U}_{\:\!\!G} =
\bs{U}_{\:\!\!G} \bs{U}_{\:\!\!G}^* = \abs{G}\:\!\bs{I}$.

With the multidimensional DFT matrix defined, we can now discuss the DFT operating on functions that depend on multiple factors. We emphasize that the entries of the forward Fourier transform reveal nonlinearities and interactions between the~$n$ influencing factors.

\begin{definition}[Discrete Fourier transform] \label{Def:DFT}
Let $\bs{f} \in L\gr{G}$, with $G$ a finite abelian group~\eqref{Eq:FiniteAbelianGroup_G}. \emph{The forward DFT\footnote{We normalize by $\abs{G}$ in the forward direction because we want the Fourier domain to constitute \emph{averages}. All of the results in this paper still hold if the reader wishes to use the form of the DFT where the $1/\abs{G}$ factor occurs in~\eqref{Def:inverseDFT} rather than~\eqref{Def:forwardDFT}; however a bit of extra attention is required. The same is true for the unitary definition of the DFT, i.e., when $1/\abs{G}^{1/2}$ is included in both~\eqref{Def:forwardDFT} and~\eqref{Def:inverseDFT}.} of~$\bs{f}$} is
\begin{equation} \label{Def:forwardDFT}
\hatbsf \,=\, \F\gr{\bs{f}} \,:=\, \frac{1}{\abs{G}} \bs{U}_{\:\!\!G} \bs{f}
\end{equation}
where~$\bs{U}_{\:\!\!G}$ in~\eqref{Eq:Multidim_U_G} is the DFT matrix associated with~$G$. For dual group~$\widehat{G}$ of~$G$ and $\hatbsf \in L\gr{\widehat{G}}$, the \emph{inverse DFT of~$\hatbsf$} is
\begin{equation} \label{Def:inverseDFT}
\bs{f} \,=\, \F^{-1}\gr{\hatbsf} \,:=\,  \bs{U}_{\:\!\!G}^* \hatbsf.
\end{equation}
Functions obeying~\eqref{Def:forwardDFT} and~\eqref{Def:inverseDFT} are said to be a \emph{Fourier pair}, summarized as
\begin{equation} \label{Def:F-pair}
\bs{f} \;\stackrel{\mathcal{F}}\rightleftharpoons\; \hatbsf
\end{equation}
where the direction of the arrows indicate the forward and inverse transforms.
\end{definition}

We assume the zeroth row/column of our DFT matrices are all $+1$'s.\footnote{There exist DFT matrices where this is not the case, e.g., whose rows or columns are permuted or are multiplied by $-1$; see~\cite[Def.~2.12]{Horadam_HadamardMatrices} and~\cite[Eq.~(13)]{somma2016quantumSimulations}.} Because of this, the zeroth element of~$\hatbsf$  will always represent the \emph{mean,~$\mu$ of~$\bs{f}$:}
\begin{equation} \label{Eq:hat{f}=mu}
\hat{f}_0 \,:=\, \frac{1}{\abs{G}} \dotprod{\bs{1}}{\bs{f}} \,=\, \frac{1}{\abs{G}} \sum_{i\in G} f_i \,=:\, \mu
\end{equation}
where $\bs{1}$ denotes the \emph{length-$\abs{G}$ vector of all ones}.

\subsubsection{Basic Fourier theorems}
There are well known properties and theorems associated with Fourier transforms. We next review those that are necessary for the subsequent sections. Assume the Fourier pairs $\bs{f} \stackrel{\mathcal{F}}\rightleftharpoons \hatbsf,\, \bs{g} \stackrel{\mathcal{F}}\rightleftharpoons \hat{\bs{g}}$ exist.

The first Fourier property is the \emph{convolution theorem}, which links convolution~\eqref{Def:Conv_finiteAbelianGroupG} in one domain to pointwise multiplication~\eqref{Def:fdotg} in the other:
$\bs{f} \conv \bs{g} \stackrel{\mathcal{F}}\rightleftharpoons \abs{G}\;\! \hatbsf \cdot \hat{\bs{g}}.$ From duality, we also have
\begin{equation} \label{Def:ConvThm_fdotg_fconvg}
\bs{f} \cdot \bs{g} \;\stackrel{\mathcal{F}}\rightleftharpoons\; \hatbsf \conv \hat{\bs{g}}.
\end{equation}
Repeated application with~$\bs{g} = \bs{f}$ yields
\begin{equation} \label{Eq:f^m=*^mfhat}
\bs{f}^m \;\stackrel{\mathcal{F}}\rightleftharpoons\;
\conv^m \hatbsf
\end{equation}
where $\conv^m \hatbsf$ denotes the \emph{autoconvolution of~$m$ copies of~$\hatbsf$}. The base cases of autoconvolution for $m=0,1$ are
\begin{equation} \label{Def:conv^0,1}
\conv^0 \!\hatbsf \,:=\, \bs{e}_0
\qquad\text{and}\qquad
\conv^1 \!\hatbsf \,:=\, \hat{\bs{f\,}}
\end{equation}
where~$\bs{e}_0$ is the zeroth standard basis element defined at the end of Section~\ref{Sec:Space_of_functions}.

The second Fourier property describes a (scaled) isometry between the two domains. \emph{Parseval/Plancherel's theorem} states that within a constant of proportionality, the inner product of two vectors is preserved under transformation: $\ip{\bs{f}}{\bs{g}} = \abs{G} \:\!\ip{\hatbsf}{\hat{\bs{g}}}$. We can use this in conjunction with~\eqref{Def:dotProduct} to derive a similar relationship for the dot product of two vectors,
\begin{equation} \label{Eq:Parseval_fg_dotprod}
\dotprod{\bs{f}}{\bs{g}} \,=\ \abs{G} \:\!\dotprod{\hatbsf}{\hat{\bs{g}}_\ominus}
\end{equation}
because $\mathcal{F}\gr{\overline{\bs{g}}} = \overline{\hat{\bs{g}}}_\ominus$ (see~\eqref{Def:fnc_f_ominus}). Of course, $\ip{\bs{f}}{\bs{g}} = \dotprod{\bs{f}}{\bs{g}}$
when~$\bs{f}$ and~$\bs{g}$ are real; on the Fourier side, this is true because of conjugate symmetry: $\hat{\bs{g}}_\ominus = \overline{\hat{\bs{g}}}$. When $G=\Int_2^n$,~\eqref{Eq:Parseval_fg_dotprod} simplifies because $\hat{\bs{g}}_\ominus = \hat{\bs{g}}$ (see Section~\ref{Sec:SpecialCase_Z2n}).

\begin{remark}[Group algebra]
Typically in harmonic analysis, the Fourier transform is viewed as providing a ring isomorphism between $\Gr{L\gr{G},+,\conv}$ and $\Gr{L\gr{\widehat{G}},+,\cdot}$, where the convolution theorem endows each space with a multiplication operation.

But the main tool we employ in Section~\ref{Sec:Moments_in_Fourier_domain} is the dual interpretation of the convolution theorem~\eqref{Def:ConvThm_fdotg_fconvg}. That is, we evaluate moments of~$\bs{f}$ defined on~$G$ by pivoting to autoconvolutions of~$\hatbsf$ defined on~$\widehat{G}$. Therefore, we are interested in the ring isomorphism between $\Gr{L\gr{G},+,\cdot}$ and
$\Gr{L\gr{\widehat{G}},+,\conv}$.
\end{remark}

\subsubsection{The $a$-diminished transform}
Suppose we subtract some constant $a$ from the elements of function~$\bs{f}$, i.e., $\bs{f}-a\bs{1}$. This is just a shift of the bias; since $\bs{1} \stackrel{\mathcal{F}}\rightleftharpoons \bs{e}_0$,
the consequence in the Fourier domain is that only the zeroth entry of~$\hatbsf$~\eqref{Eq:hat{f}=mu} will be affected.
Thus, for $a\in\C$,
\begin{equation} \label{Def:a-diminished_F-pair}
\bs{f}-a\bs{1} \;\;\stackrel{\mathcal{F}}\rightleftharpoons\;\;
\dimin{\hatbsf}{a}
\end{equation}
where the \emph{$a$-diminished transform, $\dimin{\hatbsf}{a}$}, is defined as \emph{$\hatbsf$ with its zeroth entry reduced by~$a$}:
\begin{center}
\fbox{
\addtolength{\linewidth}{-45\fboxsep}
\begin{minipage}{\linewidth}
    \begin{equation} \label{Def:a-diminished_Xform}
    \dimin{\hatbsf}{a} :=\, \hatbsf - a\bs{e}_0 \,=\,
    \left\{
      \begin{array}{ll}
        \dimin{\hat{f}}{a}_0 = \mu-a, & \quad\hbox{if $j=0$;} \\[4pt]
        \dimin{\hat{f}}{a}_j = \hat{f}_j,    & \quad\hbox{otherwise.}
      \end{array}
    \right.\;
    \end{equation}
    \vspace{1pt}
\end{minipage}
}\\[15pt]
\end{center}
For example, when $a=0$, we trivially have the original, undiminished vector
\begin{equation} \label{Eq:0-diminished_Xform}
\dimin{\hatbsf}{0} \,=\, \hatbsf \,=\, \gr{\mu,\hat{f}_1,\hat{f}_2,\ldots,\hat{f}_{\abs{G}-1}}^\top
\end{equation}
and when $a=\mu$,
\begin{equation} \label{Eq:mu-diminished_Xform}
\dimin{\hatbsf}{\mu} \,=\, \gr{0,\hat{f}_1,\hat{f}_2,\ldots,\hat{f}_{\abs{G}-1}}^\top.
\end{equation}

\subsection{Circulant matrices and subtraction tables}
We are interested in circulant matrices because they are the operators that implement convolution. As before, suppose~$G$ is the direct product of~$n$ finite cyclic groups~\eqref{Eq:FiniteAbelianGroup_G}.

Define the \emph{basic circulant permutation matrix associated with cyclic group~$\Int_N$} as $\Cbasis_{\:\!\!\Int_N} := \gr{\bs{e}_1,\bs{e}_2,\ldots,\bs{e}_{N-1},\bs{e}_0}$, i.e., the identity matrix~$\bs{I}$ with its zeroth column,~$\bs{e}_0$, shifted to the back~\cite[\S0.9.6]{hornJohnson1985MatrixAnalysis}. The \emph{multi-factor circulant permutation matrix associated with group~$G$} is $\Cbasis_{\;\!\!G} \;:=\; \Cbasis_{\;\!\!\Int_{N_1}} \otimes \Cbasis_{\;\!\!\Int_{N_2}} \otimes \cdots \otimes \Cbasis_{\;\!\!\Int_{N_n}}$. Products of powers of~$\Cbasis_{\:\!\!G}$ obey
$\Cbasis_{\:\!\!G}^i \Cbasis_{\:\!\!G}^j = \Cbasis_{\:\!\!G}^{i\oplus j}$ for $i,j\in G$, and the set $\Set{\Cbasis_{\:\!\!G}^0,\Cbasis_{\:\!\!G}^1,\ldots,\Cbasis_{\:\!\!G}^{\abs{G}-1}}$ constitutes a basis for circulant matrices defined on group~$G$, where for $k=\vect{\bcode{k}_1,\bcode{k}_2,\ldots,\bcode{k}_n}\in G$,
\begin{equation} \label{Def:PermMatrix_KG_basis}
\Cbasis_{\;\!\!G}^k \,:=\; \Cbasis_{\;\!\!\Int_{N_1}}^{\bcode{k}_1} \otimes \Cbasis_{\;\!\!\Int_{N_2}}^{\bcode{k}_2} \otimes \cdots \otimes \Cbasis_{\;\!\!\Int_{N_n}}^{\bcode{k}_n}.
\end{equation}
Then for vector~$\bs{f} = \gr{f_0,f_1,\ldots,f_{\abs{G}-1}}^\top$ with $\bs{f}:G\rightarrow\C$, the \emph{multi-factor circulant matrix associated with~$\bs{f}$} is
\begin{equation} \label{Def:CircMatrix_KG_basis}
\bs{C}_{\:\!\!\bs{f}} \,:=\; \sum_{k\in G} f_k \Cbasis_{\;\!\!G}^k.
\end{equation}
This reduces to the typical single-factor circulant matrix when $n=1$.

From its construction, the \ijth\ element of the $\abs{G}\text{-by-}\abs{G}$ circulant matrix~$\bs{C}_{\:\!\!\bs{f}}$ is a function of the difference $i \ominus j$:
\begin{equation*} \label{Def:CircMatrix_Cf_ij}
\qquad\bs{C}_{\:\!\!\bs{f}}\gr{i,j} \,:=\, f_{i \ominus j}, \qquad \text{for } i,j\in G.
\end{equation*}
Hence, its \jth\ column is~$\mathcal{S}_j\bs{f}$~\eqref{Def:ShiftOperator}, its \ith\ row is~$\mathcal{S}_i\bs{f}_{\!\ominus}^\top$~\eqref{Def:fnc_f_ominus}, and its main diagonal is~$f_0\bs{I}$, i.e.,~$\bs{C}_{\:\!\!\bs{f}}$ is composed of $n$-dimensional cyclic shifts of~$\bs{f}$ and~$\bs{f}_{\!\ominus}^\top$. (Note, some authors define a $1$-dimensional circulant matrix in terms of the difference $j \ominus i$, which results in the transpose of our~$\bs{C}_{\:\!\!\bs{f}}$, e.g., see~\cite[\S0.9.6]{hornJohnson1985MatrixAnalysis}.)

Next consider the product of two circulant matrices~$\bs{C}_{\:\!\!\bs{f}}, \bs{C}_{\;\!\!\bs{g}}$. For $i,j\in G$, its \ijth\ element is the dot product of the \ith\ row of~$\bs{C}_{\:\!\!\bs{f}}$ with the \jth\ column of~$\bs{C}_{\;\!\!\bs{g}}$:
\begin{equation} \label{Eq:C_fC_g_ij}
\bs{C}_{\:\!\!\bs{f}}\:\!\bs{C}_{\;\!\!\bs{g}}\gr{i,j} \,=\, \dotprod{\mathcal{S}_i\bs{f}_{\!\ominus}}{\mathcal{S}_j\bs{g}} \,=\, \sum_{l\in G} f_{i\ominus l} \:\! g_{l\ominus j} \,=\, \sum_{\substack{p,q\in G\\[2pt] p\oplus q = i \ominus j}} f_p \;\! g_q
\end{equation}
which is simply the dot product of~$\bs{f}_{\!\ominus}$ and~$\bs{g}$ with a relative shift of $i\ominus j$.

\subsubsection{Canonical examples}
Without a doubt, circulant matrices based on cyclic group $\Int_N$ are the most common. For function~$\bs{f} = \gr{f_0,f_1,\ldots,f_{N-1}}^\top$ with $\bs{f}:\Int_N\rightarrow\C$, its associated single-factor circulant matrix is
\begin{equation} \label{Def:CircMatrix_generalExample_ZN}
\bs{C}_{\:\!\!\bs{f}} \,\;=\;\;
\sum_{k\in \Int_N} f_k \Cbasis_{\;\!\!\Int_N}^k
\:\;=\;\:
\!\!\!\!\begin{array}{c}
      \begin{blockarray}{ccccccc}
        \scriptstyle i\backslash j & \;\scriptstyle 0 & \scriptstyle 1 & \scriptstyle \cdots & \scriptstyle N-2 & \scriptstyle N-1 \; \\[3pt]
        \begin{block}{c(cccccc)}
            \scriptstyle 0 & \;f_0 & f_{N-1}    & \cdots & f_2      & f_1 \; \\[2pt]
            \scriptstyle 1 & \;f_1 & f_0 & f_{N-1}    & \cdots & f_2 \; \\[2pt]
            \scriptstyle \vdots & \;\vdots & f_1     & f_0      & \ddots & \vdots \; \\[2pt]
            \scriptstyle N-2 & \;f_{N-2}    & \vdots & \ddots & \ddots & f_{N-1} \; \\[2pt]
            \scriptstyle N-1 & \;f_{N-1}    & f_{N-2}    & \cdots & f_1      & f_0 \; \\
        \end{block}\\[-8pt]
      \end{blockarray}\,.
\end{array}\\[-17pt]
\end{equation}

\begin{remark}
Powers of~$\Cbasis_{\;\!\!\Int_N}$ in~\eqref{Def:CircMatrix_generalExample_ZN} are responsible for the constant diagonals (parallel to the main diagonal) seen in single-factor circulant matrices. This makes their cyclic structure obvious:
each subsequent row is a cyclic shift right, and each subsequent column is cyclic shift down.

However, this behavior is generally \emph{not true} for powers of~$\Cbasis_{\;\!\!G}$  in~\eqref{Def:CircMatrix_KG_basis} for multi-factor circulant matrices.
While powers of each~$\Cbasis_{\;\!\!\Int_{N_\ell}}$ will yield constant diagonals within its respective  $N_\ell\text{-by-}N_\ell$ matrix, their $n$-fold Kronecker product in~\eqref{Def:PermMatrix_KG_basis} will not, due to the mixing property (see Remark~\ref{Rem:Mixing_property}). Instead, we will see cyclic shifts within and across sub-blocks of~$\bs{C}_{\:\!\!\bs{f}}$, acting like ``gears in a clock, each rotating at its own rate and scale.''
\end{remark}

For example, suppose $G=\Int_2^n$. The multi-factor circulant matrix associated with function~$\bs{f} = \gr{f_0,f_1,\ldots,f_{2^n-1}}^\top$ where $\bs{f}:\Int_2^n\rightarrow\C$ is
\vspace{5pt}
\begin{eqnarray} \label{Def:CircMatrix_generalExample_Z2n}
\bs{C}_{\:\!\!\bs{f}} &=& \sum_{k\in \Int_2^n} f_k \Cbasis_{\;\!\!\Int_2^n}^k \\[3pt]
&=&
\!\!\!\!\begin{array}{c}
    \begin{blockarray}{cccccccccc}
        \scriptstyle i\backslash j & \;\scriptstyle 0 & \scriptstyle 1 & \scriptstyle 2 & \scriptstyle 3 & \scriptstyle \cdots & \scriptstyle 2^n-4 & \scriptstyle 2^n-3 & \scriptstyle 2^n-2 & \scriptstyle 2^n-1 \; \\[3pt]
        \begin{block}{c(cc|cc|c|cc|cc)}
            \scriptstyle 0 & \scriptstyle \;f_0 & \scriptstyle f_1 & \scriptstyle f_2 & \scriptstyle f_3 & \cdots & \scriptstyle f_{2^n-4} & \scriptstyle f_{2^n-3} & \scriptstyle f_{2^n-2} & \scriptstyle f_{2^n-1} \; \\[6pt]
            \scriptstyle 1 & \scriptstyle \;f_1 & \scriptstyle f_0 & \scriptstyle f_3 & \scriptstyle f_2 & \cdots & \scriptstyle f_{2^n-3} & \scriptstyle f_{2^n-4} & \scriptstyle f_{2^n-1} & \scriptstyle f_{2^n-2} \; \\[1pt]
            \BAhhline{&----------}
            \scriptstyle 2 & \scriptstyle \;f_2 & \scriptstyle f_3 & \scriptstyle f_0 & \scriptstyle f_1 & \cdots & \scriptstyle f_{2^n-2} & \scriptstyle f_{2^n-1} & \scriptstyle f_{2^n-4} & \scriptstyle f_{2^n-3} \; \\[6pt]
            \scriptstyle 3 & \scriptstyle \;f_3 & \scriptstyle f_2 & \scriptstyle f_1 & \scriptstyle f_0 & \cdots & \scriptstyle f_{2^n-1} & \scriptstyle f_{2^n-2} & \scriptstyle f_{2^n-3} & \scriptstyle f_{2^n-4} \; \\[1pt]
            \BAhhline{&----------}
            \scriptstyle \vdots & \;\vdots & \vdots & \vdots & \vdots & \ddots & \vdots & \vdots & \vdots & \vdots \; \\[3pt]
            \BAhhline{&----------}
            \scriptstyle 2^n-4 & \scriptstyle \;f_{2^n-4} & \scriptstyle f_{2^n-3} & \scriptstyle f_{2^n-2} & \scriptstyle f_{2^n-1} & \cdots & \scriptstyle f_0 & \scriptstyle f_1 & \scriptstyle f_2 & \scriptstyle f_3 \; \\[6pt]
            \scriptstyle 2^n-3 & \scriptstyle \;f_{2^n-3} & \scriptstyle f_{2^n-4} & \scriptstyle f_{2^n-1} & \scriptstyle f_{2^n-2} & \cdots & \scriptstyle f_1 & \scriptstyle f_0 & \scriptstyle f_3 & \scriptstyle f_2\; \\[1pt]
            \BAhhline{&----------}
            \scriptstyle 2^n-2 & \scriptstyle \;f_{2^n-2} & \scriptstyle f_{2^n-1} & \scriptstyle f_{2^n-4} & \scriptstyle f_{2^n-3} & \cdots & \scriptstyle f_2 & \scriptstyle f_3 & \scriptstyle f_0 & \scriptstyle f_1 \; \\[6pt]
            \scriptstyle 2^n-1 & \scriptstyle \;f_{2^n-1} & \scriptstyle f_{2^n-2} & \scriptstyle f_{2^n-3} & \scriptstyle f_{2^n-4} & \cdots & \scriptstyle f_3 & \scriptstyle f_2 & \scriptstyle f_1 & \scriptstyle f_0 \; \\
    \end{block}\\[-6pt]
    \end{blockarray}
\end{array} \nonumber
\end{eqnarray}\\[-25pt]
where $\Cbasis_{\;\!\!\Int_2^n} = \Cbasis_{\;\!\!\Int_2} \otimes \Cbasis_{\;\!\!\Int_2} \otimes \cdots \otimes \Cbasis_{\;\!\!\Int_2}$ ($n$ copies). The rightmost copy of~$\Cbasis_{\;\!\!\Int_2}$ is responsible for the \emph{micro} $\Int_2$-cycles within each of the delineated $2\text{-by-}2$ sub-blocks ($16$ of the $\gr{2^{n-1}}^2$ sub-blocks are shown). The second rightmost copy of~$\Cbasis_{\;\!\!\Int_2}$ is responsible for the $\Int_2$-cycles of the $2\text{-by-}2$ sub-blocks across each of the $4\text{-by-}4$ sub-quadrants ($4$ of the $\gr{2^{n-2}}^2$ sub-quadrants are shown), and so on. Finally, the leftmost copy of~$\Cbasis_{\;\!\!\Int_2}$ is responsible for the \emph{macro} $\Int_2$-cycle of the four $2^{n-1}\text{-by-}2^{n-1}$ quadrants across the whole $2^n\text{-by-}2^n$ matrix.

\subsubsection{Subtraction tables} \label{Sec:SubtractionTables}
We now introduce the \emph{index subtraction table~$\bs{\mathcal{A}}_{\,G}$ associated with group~$G$}, simply defined as the $\abs{G}\text{-by-}\abs{G}$ matrix with
\begin{equation} \label{Def:SubtractionMatrix_A_G_ij}
\qquad\bs{\mathcal{A}}_{\,G}\gr{i,j} \::=\; i \ominus j, \qquad \text{for } i,j\in G.
\end{equation}
Given any circulant matrix~$\bs{C}_{\:\!\!\bs{f}}$ with~$\bs{f}$ defined on~$G$, the index subtraction table~$\bs{\mathcal{A}}_{\,G}$ is easily obtained by mapping each cell containing~$f_k$ to the value~$k$, and vice versa:
\begin{equation} \label{Def:C_f<-->A_G}
\bs{C}_{\:\!\!\bs{f}} \,\; \stackrel{f_k \leftrightarrow k}{\longleftrightarrow} \,\; \bs{\mathcal{A}}_{\,G}
\end{equation}
(there are algebraic methods to construct~$\bs{\mathcal{A}}_{\,G}$ directly from~$G$, but we do not present them here). Clearly,~$\bs{C}_{\:\!\!\bs{f}}$ and~$\bs{\mathcal{A}}_{\,G}$ possess the same patterns and structure. In particular, every index $k\in G$ appears exactly once in each row and column of~$\bs{\mathcal{A}}_{\,G}$.

The utility of a given subtraction table~$\bs{\mathcal{A}}_{\,G}$ is that it can help to understand the group structure of $G$. This may not be possible by simply looking at a circulant matrix~$\bs{C}_{\:\!\!\bs{f}}$ defined on~$G$ because its $\set{f_k}$ values are arbitrary; the patterns will be obscured by small or large dynamic ranges. But the elements, $k \in G$, in~$\bs{\mathcal{A}}_{\,G}$ are unique, and their decimal equivalents have an ordered, linear dynamic range.

For example, consider the visual depictions of the $64\text{-by-}64$ subtraction tables~$\bs{\mathcal{A}}_{\,\Int_{64}}$ and~$\bs{\mathcal{A}}_{\,\Int_2^6}$ in Figure~\ref{Fig:SubtractionTable_Z64_Z2^6}; these are associated with the circulant matrices in~\eqref{Def:CircMatrix_generalExample_ZN} with $N=64$, and in~\eqref{Def:CircMatrix_generalExample_Z2n} with $n=6$, respectively. The group structures are immediately clear --- the simple pattern in~$\bs{\mathcal{A}}_{\,\Int_{64}}$ is familiar and intuitive, while the contrapuntal patterns in~$\bs{\mathcal{A}}_{\,\Int_2^6}$ at all $6$ scales require some contemplative thought. Further, all subtraction tables defined on $G=\Int_N$ will look similar to~$\bs{\mathcal{A}}_{\,\Int_{64}}$, and all of those defined on $G = \Int_2^n$ will look essentially the same as~$\bs{\mathcal{A}}_{\,\Int_2^6}$ due to their multiscale dyadic structure.

\begin{figure}[!tb]
\centering
\vspace{10pt}
\hspace{1pt} Subtraction table $\bs{\mathcal{A}}_{\,\Int_{64}}$ \hspace{91pt} Subtraction table $\bs{\mathcal{A}}_{\,\Int_2^6}$\quad~{}\\
\includegraphics[scale=0.3, trim={60 120 50 25mm},clip] 
{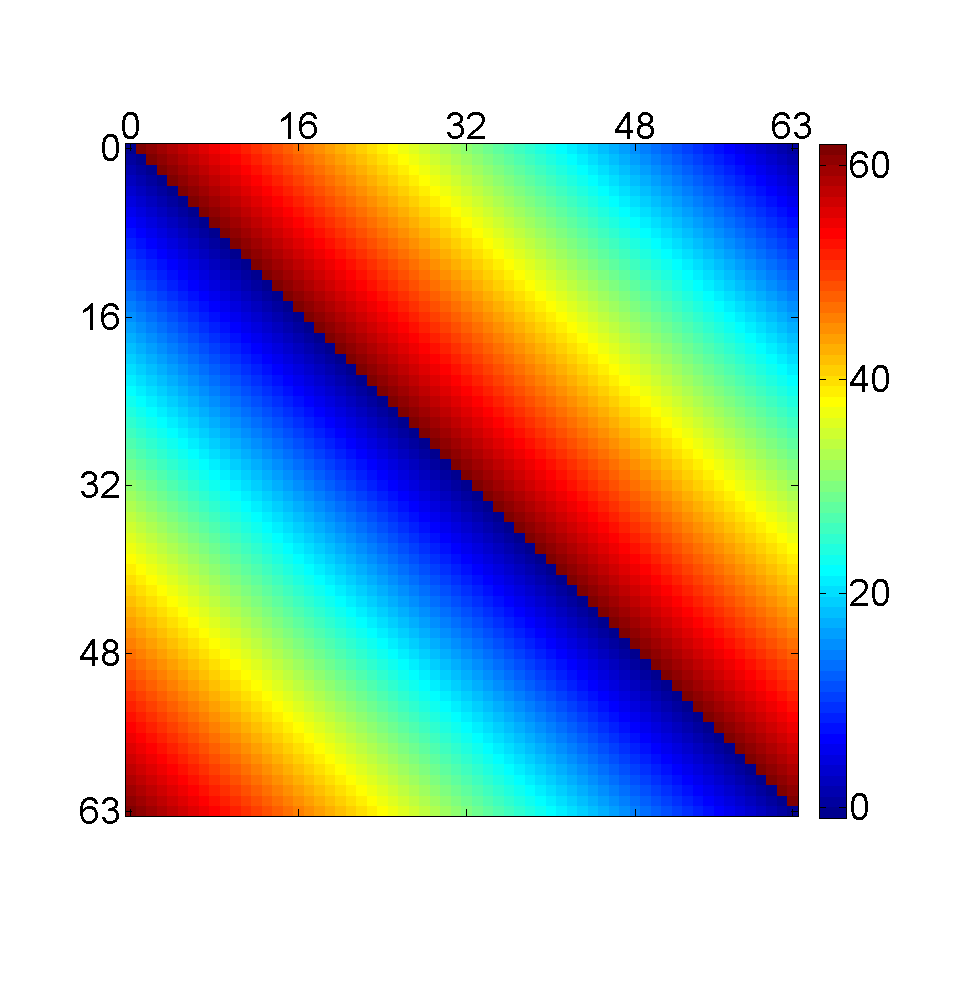}
\qquad
\includegraphics[scale=0.3, trim={60 120 50 25mm},clip] 
{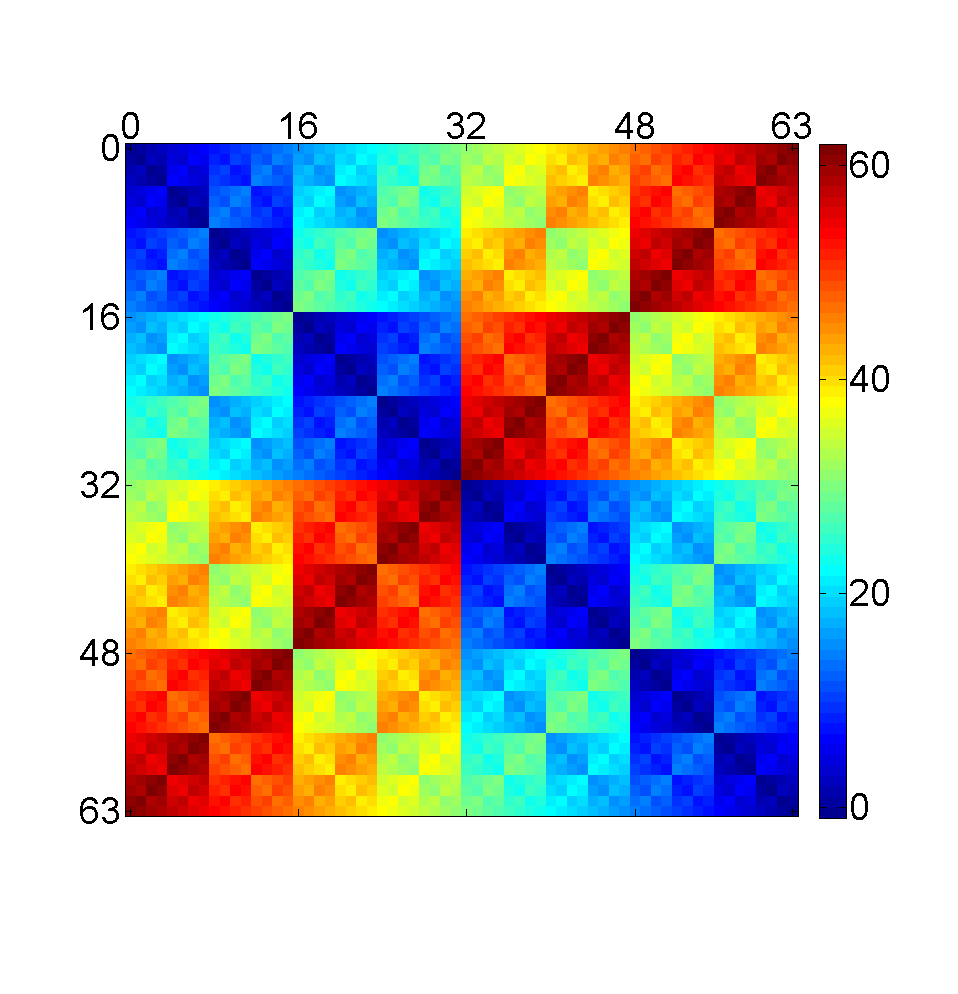}
\caption{Heat maps of~$\bs{\mathcal{A}}_{\,\Int_{64}}$ (left) and~$\bs{\mathcal{A}}_{\,\Int_2^6}$ (right). Both subtraction tables are $64\text{-by-}64$, and each row and column contain the integers $0,1,\ldots,63$, yet they look completely different due to their respective group structures. Notice, $\bs{\mathcal{A}}_{\,\Int_{64}}$ has just $n=1$ macro $\Int_{64}$-cycle, while~$\bs{\mathcal{A}}_{\,\Int_2^6}$ has $\Int_2$-cycles across all sub-quadrants at each of the $n=6$ scales. \label{Fig:SubtractionTable_Z64_Z2^6}}
\end{figure}

More can be said about~$\bs{\mathcal{A}}_{\,\Int_2^n}$. Of significant note is the fact that $k=\ominus k$ for each $k\in\Int_2^n$. Thus, for an $n$-bit string, bitwise subtraction is the same as bitwise addition (bitwise \bcode{XOR} operation). As such, $\bs{\mathcal{A}}_{\,\Int_2^n}$ is both a subtraction and an addition table, and $\bs{\mathcal{A}}_{\,\Int_2^n} = \bs{\mathcal{A}}_{\,\Int_2^n}^\top$. These tables also represent \emph{nimber addition} in combinatorial game theory~\cite{berlekamp2000WinningWaysMathematicalPlays}.

\subsection{Implementation of convolution}
We now address the obvious connection between circulant matrices and circular convolution. Assume we have the $\abs{G}\text{-by-}\abs{G}$ circulant matrices~$\bs{C}_{\:\!\!\bs{f}},\bs{C}_{\;\!\!\bs{g}}$ associated with vectors~$\bs{f},\bs{g}$ defined on finite abelian group~$G$.

Let us reexamine the \ijth\ element of the product of~$\bs{C}_{\:\!\!\bs{f}}$ and~$\bs{C}_{\;\!\!\bs{g}}$. Setting ${j=0}$ in~\eqref{Eq:C_fC_g_ij}, we see that the first summation (with dummy variable~$l$) is precisely the definition of convolution in~\eqref{Def:Conv_finiteAbelianGroupG} (with dummy variable~$j$). Therefore, the zeroth column of~$\bs{C}_{\:\!\!\bs{f}}\:\!\bs{C}_{\;\!\!\bs{g}}$ represents $\bs{f} \conv \bs{g}$. But the zeroth column of~$\bs{C}_{\:\!\!\bs{g}}$ is vector~$\mathcal{S}_0\bs{g}=\bs{g}$ itself. Thus, the convolution of two vectors can be realized by a matrix-vector product:
\begin{equation} \label{Eq:Conv_f*g=C_fg}
\bs{f} \conv \bs{g} \;=\; \bs{C}_{\;\!\!\bs{f}} \;\! \bs{g} \;=\;
\bs{C}_{\bs{g}} \:\! \bs{f} \;
\end{equation}
where we used commutativity, $\bs{f} \conv \bs{g} = \bs{g} \conv \bs{f}$, for the second equality. From this, it is easy to show when $\bs{g}=\bs{f}$ and for any integer $m\geq0$ that
\begin{equation} \label{Eq:C_*^mf=C_f^m}
\bs{C}_{\gr{\:\!\!\conv^m\:\!\!\bs{f}}} \,=\, \bs{C}_{\:\!\!\bs{f}}^m
\end{equation}
(note, the autoconvolution base cases in~\eqref{Def:conv^0,1} are of the same form for $\conv^0\bs{f}$ and $\conv^1\bs{f}$).

\begin{remark}
Our generalized circulant matrices~\eqref{Def:CircMatrix_KG_basis} are a form of \emph{currying}~\cite{curry1958combinatory}. The underlying group provides a matrix pattern seen in the associated subtraction table~\eqref{Def:SubtractionMatrix_A_G_ij}; a function~$\bs{f}$ inserted according to this pattern~\eqref{Def:C_f<-->A_G} yields a circulant matrix~$\bs{C}_{\:\!\!\bs{f}}$ as the linear transformation which effects the operation of convolution~\eqref{Eq:Conv_f*g=C_fg} in the algebra of functions on the underlying group.\footnote{The process of reducing a binary product to a family of operators predates Curry's work~\cite{Frege1893FREGDA2}. It is a foundation of Moses Sch\"onfinkel's theory of combinators~\cite{schonfinkel1924buildingBlocksLogic,wolfram2021combinators}. Baez~\cite{baez2011aRosettaStone} discusses the extension to category theory.}
\end{remark}

\section{Evaluating moments in the Fourier domain} \label{Sec:Moments_in_Fourier_domain}
As a reminder to the reader, we are interested in obtaining the \emph{population statistics} of~$\bs{f}$, not its sample statistics, via the moments in Section~\ref{Sec:Gen_raw_central_momennts}. Rather than work in the observation domain, we wish instead to evaluate these moments directly from the (often more concise) Fourier domain description,~$\hatbsf$. Below, we present how to express the \mth\ moment,~$\dimin{\mu}{a}_m$, and the variance,~$\sigma^2$, in the Fourier domain, along with a curious annihilation property.

\subsection{The \mth\ moment from Fourier coefficients}
Recall the \mth\ general moment of~$\bs{f}$ from~\eqref{Def:mthMoment_a} and assume functions $\bs{f} \! \stackrel{\mathcal{F}}\rightleftharpoons \! \hatbsf$ are a Fourier pair~\eqref{Def:F-pair}. Notice that the summation can be rewritten using a dot product~\eqref{Def:dotProduct}. Then we can use Parseval/Plancherel's theorem~\eqref{Eq:Parseval_fg_dotprod}, the $a$-diminished transform~\eqref{Def:a-diminished_F-pair}, and the (auto)convolution theorem~\eqref{Eq:f^m=*^mfhat} to express this in the Fourier domain. Hence, for integer $m\geq0$ and some integer $0 \leq p \leq m$, the \mth\ general moment of~$\bs{f}$ about point~$a\in\C$ in terms of its Fourier coefficients is
\begin{center}
\fbox{
\addtolength{\linewidth}{-55\fboxsep}
 \begin{minipage}{\linewidth}
    \vspace{-5pt}
    \begin{eqnarray}
    \dimin{\mu}{a}_m\gr{\bs{f}} &:=& \frac{1}{\abs{G}} \sum_{i\in G} \gr{f_i-a}^m \nonumber \\
    &=& \frac{1}{\abs{G}} \Dotprod{\gr{\bs{f} - a\bs{1}}^{m-p}}{\gr{\bs{f} - a\bs{1}}^p} \;\; \nonumber\\[5pt]
    &=& \Dotprod{\!\conv^{m-p} \!\dimin{\hatbsf}{a}}{\conv^p \dimin{\hatbsf}{a}_{\!\ominus}}. \label{Eq:mu_m^a=g*gtg*g}
    \end{eqnarray}
    \vspace{-15pt}
 \end{minipage}
}\\[15pt]
\end{center}

We will be almost exclusively interested in raw and central moments~\eqref{Def:mthRawCentralMoments}. Viewing these cases from the Fourier side, we see that the ``energy'' of the mean, $\mu$, is fully present in the \mth\ raw moment~$\mu'_m\gr{\bs{f}}$ when~$\dimin{\hatbsf}{0}$~\eqref{Eq:0-diminished_Xform} is substituted into~\eqref{Eq:mu_m^a=g*gtg*g}, and is fully absent in the \mth\ central moment~$\mu_m\gr{\bs{f}}$ when~$\dimin{\hatbsf}{\mu}$~\eqref{Eq:mu-diminished_Xform} is used, as desired.

\subsubsection{Agreement with the base cases}
The base cases of moments, derived in the original domain~\eqref{Eq:mthMoment_base_cases}, can also be viewed through our Fourier lens via the base cases of autoconvolution in~\eqref{Def:conv^0,1}. For $m=0$ we must have $p=0$ in~\eqref{Eq:mu_m^a=g*gtg*g}, so the zeroth moment is unity because for any $a\in\C$
$$
\dimin{\mu}{a}_0\gr{\bs{f}} \,=\, \Dotprod{\!\conv^0 \!\dimin{\hatbsf}{a}}{\conv^0 \dimin{\hatbsf}{a}_{\!\ominus}}
\,=\, \dotprod{\bs{e}_0}{\bs{e}_0} \,=\, 1.
$$
For $m=1$ and $p=0$ in~\eqref{Eq:mu_m^a=g*gtg*g}, the first general moment centered at $a$ is
\begin{equation} \label{Eq:AgreeBaseCase_m=1}
\dimin{\mu}{a}_1\gr{\bs{f}} \,=\, \Dotprod{\!\conv^1 \!\dimin{\hatbsf}{a}}{\conv^0 \dimin{\hatbsf}{a}_{\!\ominus}}
\,=\, \dotprod{\dimin{\hatbsf}{a}}{\bs{e}_0} \,=\, \dimin{\hat{f}}{a}_0 \,:=\, \mu - a
\end{equation}
from~\eqref{Def:a-diminished_Xform} (using $p=1$ yields the same). These both agree with~\eqref{Eq:mthMoment_base_cases}, as expected.

\subsubsection{The variance from Fourier coefficients}
We cannot use~\eqref{Eq:mu_m^a=g*gtg*g} when evaluating the variance~\eqref{Def:Variance} since it involves the sum of magnitudes squared. However, the traditional form of Parseval/Plancherel's theorem is appropriate along with $a=\mu$ in~\eqref{Def:a-diminished_F-pair} and~\eqref{Eq:mu-diminished_Xform}:
\begin{equation}  \label{Eq:variance_|fj|^2}
\sigma^2\gr{\bs{f}} \;:=\; \frac{1}{\abs{G}} \sum_{i\in G} \abs{f_i-\mu}^2
\;=\; \frac{1}{\abs{G}} \Norm{\bs{f} - \mu\bs{1}}^2
\;=\; \Norm{\dimin{\hatbsf}{\mu}}^2
\;=\;
\!\sum_{j\in G\backslash0} \Abs{\hat{f}_j}^2.
\end{equation}
Here, $G\backslash0$ is shorthand for~$G\backslash\set{0}$, the removal of element~$0$ from the set~$G$.

\subsection{A surprising filtering property}
Consider the \mth\ raw moment of a function~$\bs{f}$. Carrying out the dot product in~\eqref{Eq:mu_m^a=g*gtg*g} with Fourier transform~$\hatbsf$ in~\eqref{Eq:0-diminished_Xform}
results in terms that contain $m$ coefficients: $\Set{\hat{f}_{j_1} \hat{f}_{j_2} \cdots \hat{f}_{j_m}}$. This is to be expected since the autoconvolution~$\conv^{m-p} \hatbsf$ will have terms with ${m-p}$ coefficients, and the autoconvolution~$\conv^p {\hatbsf}_{\!\ominus}$ will have terms with~$p$ coefficients, so their dot product will result in terms with precisely~$m$ Fourier coefficients (however, it may be that some $\hat{f}_j$ occur more than once in a given term).

Surely though, not all combinations of Fourier coefficients should influence the \mth\ moment. For example, consider the group $G=\Int_3\times\Int_2$ with elements
$$
\vect{\bcode{0},\bcode{0}}=0,\quad \vect{\bcode{0},\bcode{1}}=\bcode{1},\quad \vect{\bcode{1},\bcode{0}}=2,\quad \vect{\bcode{1},\bcode{1}}=3,\quad \vect{\bcode{2},\bcode{0}}=4,\quad \vect{\bcode{2},\bcode{1}}=5
$$
and a function~$\bs{f}: \Int_3\times\Int_2 \rightarrow \C$. Suppose we are interested in the function's third raw moment, $\mu'_{3}$. Since ${\abs{\Int_3\times\Int_2}=6}$, the number of possible terms with three Fourier coefficients is $\Abs{\Set{\hat{f}_{j_1} \hat{f}_{j_2} \hat{f}_{j_3}}} = \binom{6+3-1}{3} = \binom{8}{3} = 56$ (the number of combinations with replacement), where $\binom{n}{k}$ is the \emph{\kth\ binomial coefficient}. Yet evaluating~\eqref{Eq:mu_m^a=g*gtg*g} with $a=0$ and $m=3$, we get
\begin{eqnarray*}
\mu'_{3}\gr{\bs{f}} &=& \hat{f}_0^3 + \hat{f}_2^3 + \hat{f}_4^3 + 3 \gr{\hat{f}_0 \hat{f}_1^2 + \hat{f}_2 \hat{f}_3^2 + \hat{f}_{\alertblue{4}} \hat{f}_{\alertblue{5}}^2} \nonumber\\
&& \;\; +\: 6 \gr{\hat{f}_0 \hat{f}_2 \hat{f}_4 + \hat{f}_1 \hat{f}_3 \hat{f}_4 + \hat{f}_1 \hat{f}_2 \hat{f}_5 + \hat{f}_0 \hat{f}_3 \hat{f}_5}
\end{eqnarray*}
which has just~$10$ unique terms. What makes these terms special, and the other~$46$ not? The answer lies in the underlying group operation: the indices of each $3$-coefficient term annihilate under modulo addition. For example, consider the $\hat{f}_{\alertblue{4}}\hat{f}_{\alertblue{5}}^2 = \hat{f}_{\alertblue{4}}\hat{f}_{\alertblue{5}}\hat{f}_{\alertblue{5}}$ term; summing these three indices yields
$$
4\oplus5\oplus5 = \vect{\bcode{2},\bcode{0}} \oplus \vect{\bcode{2},\bcode{1}} \oplus \vect{\bcode{2},\bcode{1}} = \vect{\bcode{6}\!\!\!\!\pmod{3},\bcode{2}\!\!\!\!\pmod{2}} = \vect{\bcode{0},\bcode{0}} = 0.
$$

This index annihilation property can be viewed as a sort of ``filtering'' in the sense that it restricts which Fourier coefficients can and cannot group together --- in fact, this holds true for every term in any~\mth\ moment of~$\bs{f}$. We present it next as a theorem.

\begin{thm}[$m$-Coefficient, Index Annihilation] \label{Thm:mthMoment_a}
Assume function~$\bs{f}:G\rightarrow\C$ with~$G$ a finite abelian group~\eqref{Eq:FiniteAbelianGroup_G}, and Fourier pair $\bs{f} \stackrel{\mathcal{F}}\rightleftharpoons \hatbsf$~\eqref{Def:F-pair}. Then for integer $m\geq1$, the \mth\ general moment of~$\bs{f}$ about point~$a\in\C$ is of the form
\begin{equation} \label{ThmEq:mthMoment_a}
\dimin{\mu}{a}_m\gr{\bs{f}}
\,\;=\;\; \!\!\sum_{\substack{j_q\:\!\in\:\!G\\[2pt]
\bigoplus_{q=1}^m j_q = 0}} \dimin{\hat{f}}{a}_{j_1} \dimin{\hat{f}}{a}_{j_2} \cdots \dimin{\hat{f}}{a}_{j_m}
\end{equation}
where the $a$-diminished transform~$\dimin{\hatbsf}{a}$
is defined in~\eqref{Def:a-diminished_Xform}.
That is, each term consists of a product of precisely $m$ Fourier coefficients, $\dimin{\hat{f}}{a}_{j_1} \dimin{\hat{f}}{a}_{j_2} \cdots \dimin{\hat{f}}{a}_{j_m}$, such that the modulo summation of the ${j_q}$'s is zero.\footnote{Index $j_q\in G$ identifies the \qth\ element of the Fourier transform. It should not be confused with the \qth\ position of an arbitrary $n$-tuple $j = \vect{\bcode{j}_1,\bcode{j}_2,\ldots,\bcode{j}_n} \in G$, where $\bcode{j}_q\in\Int_{N_q}$.}
\end{thm}

\begin{proof}
Fix some integer $m\geq1$. Examining the right-hand side of~\eqref{Eq:mu_m^a=g*gtg*g} with any integer $0\leq p\leq m$, we see from~\eqref{Eq:C_*^mf=C_f^m} that $\conv^{m-p} \dimin{\hatbsf}{a}$ is just column $j=0$ of~$\bs{C}_{\!\!\dimin{\hatbsf}{a}}^{m-p}$, and $\conv^p \dimin{\hatbsf}{a}_{\!\ominus}$ is row $i=0$ of~$\bs{C}_{\!\!\dimin{\hatbsf}{a}}^{p}$. Thus we have $\dimin{\mu}{a}_m\;\!\!\gr{\bs{f}} = \bs{C}_{\!\!\dimin{\hatbsf}{a}}^m \gr{0,0}$.
Here, the relative shift is $i \ominus j=0$.
It follows that the main diagonal of $\bs{C}_{\!\!\dimin{\hatbsf}{a}}^m$ is
\begin{equation} \label{Eq:C_fC_f_kk-mu_m}
\bs{C}_{\!\!\dimin{\hatbsf}{a}}^m \gr{k,k} \,=\, \dimin{\mu}{a}_m\;\!\!\gr{\bs{f}}
\end{equation}
because for any $i=j=k\in G$, the relative shift will also be $i \ominus j=0$.

At the same time, we can invoke~\eqref{Eq:C_fC_g_ij} recursively to claim for any $k,l\in G$ that the \klth\ element of $\bs{C}_{\!\!\dimin{\hatbsf}{a}}^m = \bs{C}_{\!\!\dimin{\hatbsf}{a}}\bs{C}_{\!\!\dimin{\hatbsf}{a}}\cdots\bs{C}_{\!\!\dimin{\hatbsf}{a}}$ is
$$
\bs{C}_{\!\!\dimin{\hatbsf}{a}}^m \gr{k,l} = \sum_{j_q\in G} \hat{f}_{j_1}^{\gr{a}} \hat{f}_{j_2}^{\gr{a}} \cdots \hat{f}_{j_m}^{\gr{a}} \quad \text{s.t.} \quad j_1 \oplus j_2 \oplus \cdots \oplus j_m = k \ominus l.
$$
But the main diagonal enforces $k=l$, which yields~\eqref{ThmEq:mthMoment_a}.
\end{proof}

\begin{remark} \label{Rem:n-tupple_Annihilation}
Consider the $j_q$ indices of Theorem~\ref{Thm:mthMoment_a} as $n$-tuples.
Then the index annihilation condition, $j_1 \oplus j_2 \oplus \cdots \oplus j_m = 0$, means that the sum of the $n$-tuples in position~$\ell$ equals ${\bcode{0} \!\pmod{N_\ell}}$, for all $\ell=1,2,\ldots,n$.
\end{remark}

\subsubsection{Implications of Theorem~\ref{Thm:mthMoment_a}}
Let us derive the first and second moments of an arbitrary function using Theorem~\ref{Thm:mthMoment_a}. For $m=1$, the theorem demands that each term in~\eqref{ThmEq:mthMoment_a} possesses \emph{only one} Fourier coefficient,~$\hat{f}_{j}^{\gr{a}}$. But the only coefficient that satisfies the index annihilation constraint is $j=0$, hence
$$
\dimin{\mu}{a}_1\gr{\bs{f}} \,=\, \hat{f}_{0}^{\gr{a}} \,:=\, \mu-a.
$$
This is the same as~\eqref{Eq:AgreeBaseCase_m=1}, but the steps to obtain~$\hat{f}_{0}^{\gr{a}}$ are completely different.

For $m=2$, the terms can contain \emph{only two} Fourier coefficients, so the second moment must be of the form
$$
\dimin{\mu}{a}_2\gr{\bs{f}} \,=\, \sum_{j\in G} \:\! \hat{f}_{j}^{\gr{a}} \hat{f}_{\ominus j}^{\gr{a}}
$$
because of the index annihilation constraint.
If~$\bs{f}$ is a real-valued function, then its Fourier transform  will be conjugate symmetric, so~${\hatbsf}_{\!\ominus} = \overline{\hatbsf\,}$ and each term will reduce to~$\Abs{\hat{f}_{j}^{\gr{a}}}^2$. When $a=0$, this coincides with Parseval/Plancherel's conservation of energy theorem, and when $a=\mu$, this yields the variance~$\sigma^2\gr{\bs{f}}$~\eqref{Eq:variance_|fj|^2}.

We are mostly concerned with the raw and central moments of~$\bs{f}$, so setting $a=0$ and $a=\mu$ in~\eqref{ThmEq:mthMoment_a}, they are for $m\geq1$:
\begin{center}
\fbox{
\addtolength{\linewidth}{-70\fboxsep}
 \begin{minipage}{\linewidth}
    $$
    \mu'_m\gr{\bs{f}} \,=\; \!
    \sum_{\substack{j_q\:\!\in\:\!G\\[2pt]
    \bigoplus_{q=1}^m j_q = 0}} \hat{f}_{j_1} \hat{f}_{j_2} \cdots \hat{f}_{j_m}
    $$
 \end{minipage}
}\\[15pt]
\end{center}
and
\begin{center}
\fbox{
\addtolength{\linewidth}{-65\fboxsep}
 \begin{minipage}{\linewidth}
    \begin{equation} \label{ThmEq:mthMoment_mu}
    \!\!\mu_m\gr{\bs{f}} \,=\; \!
    \sum_{\substack{j_q\:\!\in\:\!G\backslash0\\[2pt] \bigoplus_{q=1}^m j_q = 0}} \hat{f}_{j_1} \hat{f}_{j_2} \cdots \hat{f}_{j_m}. \;
    \end{equation}
 \end{minipage}
}\\[15pt]
\end{center}

\subsection{Implementation and computational complexity of~$\dimin{\mu}{a}_m\gr{\bs{f}}$} \label{Sec:ComputationalComplexity}
In this section we discuss the computational effort to obtain~$\dimin{\mu}{a}_m\gr{\bs{f}}$ from the $a$-diminished transform~$\dimin{\hatbsf}{a}$, whether as a scalar value or as an algebraic expression. Blindly carrying out the summation in~\eqref{ThmEq:mthMoment_a} requires $m-1$ nested summations, similar to~\eqref{Eq:Nikias_mth-order_moment-zeroLags}.
For each combination of $j_1,j_2,\ldots,j_{m-1} \in G$, a unique $n$-tuple $j_m$~\eqref{Eq:j_m_resonance} must be calculated, which takes $\mathcal{O}\Gr{n\gr{m-1}}$ operations. There are~$\abs{G}^{m-1}$ different combinations to evaluate, so altogether this has a forbidding $\mathcal{O}\Gr{n\gr{m-1}\abs{G}^{m-1}}$ complexity.

However, this does not leverage our autoconvolutional approach, which does not employ any combinatorics. If~$\abs{G}$ is not too large, or if computing resources are not a concern, then it is easiest to just evaluate~$\bs{C}_{\!\!\dimin{\hatbsf}{a}}^m$ and then find~$\dimin{\mu}{a}_m\gr{\bs{f}}$ along its main diagonal~\eqref{Eq:C_fC_f_kk-mu_m}.

Otherwise, there are symmetries and tricks to implement~\eqref{Eq:mu_m^a=g*gtg*g} efficiently. First, we can always generate vector $\conv^p \dimin{\hatbsf}{a}_{\!\ominus}$ from vector~$\conv^{p} \dimin{\hatbsf}{a}$ by a permutation or a lookup table, which costs $\mathcal{O}\Gr{\abs{G}}$ operations. So we only need to be concerned with autoconvolutions of~$\dimin{\hatbsf}{a}$. Next, given~$\conv^{m-p} \dimin{\hatbsf}{a}$ and~$\conv^p \dimin{\hatbsf}{a}_{\!\ominus}$, the final step of their dot product in~\eqref{Eq:mu_m^a=g*gtg*g} requires only~$\mathcal{O}\Gr{\abs{G}}$ operations. For example, the second moment is just $\dimin{\mu}{a}_2\gr{\bs{f}} = \Dotprod{\dimin{\hatbsf}{a}}
{\dimin{\hatbsf}{a}_{\!\ominus}}$.

For pairs of moments, $m=2k$ and $m-1$ with $k\geq2$, recursively create vector $\conv^k \dimin{\hatbsf}{a} = \bs{C}_{\!\!\dimin{\hatbsf}{a}}^{k-1} \dimin{\hatbsf}{a} = \bs{C}_{\!\!\dimin{\hatbsf}{a}} \Gr{\cdots \bs{C}_{\!\!\dimin{\hatbsf}{a}} \Gr{\bs{C}_{\!\!\dimin{\hatbsf}{a}} \dimin{\hatbsf}{a}}}$, whose ${k-1}$ matrix-vector products require $\mathcal{O}\Gr{\gr{k-1}\abs{G}^2}$ operations. Then, we immediately get
$$
\dimin{\mu}{a}_{m-1}\gr{\bs{f}} = \Dotprod{\!\conv^{k} \!\dimin{\hatbsf}{a}}{\conv^{k-1} \dimin{\hatbsf}{a}_{\!\ominus}}
\qquad\text{and}\qquad
\dimin{\mu}{a}_m\gr{\bs{f}} = \Dotprod{\!\conv^{k} \!\dimin{\hatbsf}{a}}{\conv^k \dimin{\hatbsf}{a}_{\!\ominus}}.
$$
With~$m$ even, this requires $\mathcal{O}\Gr{\gr{\frac{m}{2}-1}\abs{G}^2}$ operations. Yet in the process, we also get all moments~$\dimin{\mu}{a}_{m'}\gr{\bs{f}}$ for $3\leq m'\leq m$. Further, if $\hatbsf$ is $s$-sparse, then this drops to $\mathcal{O}\Gr{\gr{\frac{m}{2}-1}\abs{G}s}$. If there is any other symmetry, structure, etc.~of the support of~$\hatbsf$, then additional reduction in the number of operations may be possible.

\section{Special case of $G=\Int_2^n$} \label{Sec:SpecialCase_Z2n}
Let us look more closely at the group $G=\Int_2^n$. This occurs when each $\Int_{N_\ell} = \Int_2$ in~\eqref{Eq:FiniteAbelianGroup_G}, i.e., each of the $n$ factors are binary. In this case, the group can be represented by an $n$-dimensional Boolean cube, with each of its $2^n$ vertices identified by a different element of~$\Int_2^n$. For a function~$\bs{f}$ defined on~$\Int_2^n$, its \ith\ value~$f_i$, stored at the \ith\ vertex, is the response to the \ith\ state of the~$n$ binary factors: $i=\vect{\bcode{i}_1, \bcode{i}_2, \ldots, \bcode{i}_n}$, with each $\bcode{i}_\ell\in\Int_2$.

There are endless applications of~$\Int_2^n$ in the natural and social sciences. Note that the~$n$ binary variables affecting function~$\bs{f}$ may be discrete (two different states, such as \bcode{A}/\bcode{B}) or continuous (\bcode{Low}/\bcode{High} values on a continuum), as well as binary nominal data. A short list includes graph/hypergraph theory, error-correcting codes, cryptography, solid state physics, computational chemistry, biology and genetics, combinatorial designs including design of experiments, social choice theory and voting, and economics. Section~\ref{Sec:Applications} contains two examples.

As mentioned in Section~\ref{Sec:SubtractionTables}, bitwise subtraction is the same as addition, thus all elements of $\Int_2^n$ are their own additive inverse. This means for any function~$\bs{f}$ defined on~$\Int_2^n$, we have $\bs{f} = \bs{f}_{\!\ominus}$ and $\hatbsf = {\hatbsf}_{\!\ominus}$. As such, when evaluating the \mth\ moment~$\dimin{\mu}{a}_m\gr{\bs{f}}$, the dot product in~\eqref{Eq:mu_m^a=g*gtg*g} simplifies (see Section~\ref{Sec:CentralMoments_Z2n}).

\subsection{Fourier transform on~$\Int_2^n$: Interactions}
The DFT matrix~\eqref{Eq:Multidim_U_G} associated with~$\Int_2^n$ is~$\bs{U}_{\Int_2^n}$, the \emph{Sylvester-Hadamard\,\footnote{These matrices are often incorrectly referred to as ``Walsh-Hadamard.'' The rows/columns of Walsh-Hadamard matrices are ordered in terms of their \emph{sequency}, while those of Sylvester-Hadamard are described as being in their \emph{natural} order. See~\cite[Def.~2.3, Eq.~(3.7)]{Horadam_HadamardMatrices} and~\cite[Footnote~5]{doro2021fourierTransQuantTraitCS}.} matrix of order~$2^n$}. The Fourier transform~$\hatbsf$, defined on the dual group~$\widehat{G}=\Int_2^n$, should be viewed as \emph{interactions between the~$n$ factors}. This group is represented by its own $n$-dimensional Boolean cube.

The \jth\ Fourier coefficient~$\hat{f}_j$ is stored at the \jth\ vertex, and~$\abs{\hat{f}_j}$ quantifies the strength of the \jth\ interaction, with $j = \vect{\bcode{j}_1, \bcode{j}_2, \ldots, \bcode{j}_n}$. But now each $\bcode{j}_\ell\in\Int_2$ is an \emph{indicator function} identifying whether or not the \ellth\ factor participates in the \jth\ interaction. Equivalently, we can represent index $j$ by the \emph{set of participating factors}. So in~$\Int_2^n$ there are three equivalent ways to represent an index: \emph{decimal, binary, set}. For example, for $n=3$ factors represented by the vertices labeled $\ell = 1,2,3$, the index $j = 6 = \vect{\bcode{0},\bcode{1},\bcode{1}} = \set{2,3}$.

For each $j\in\Int_2^n$, its \emph{degree} is simply the number of participating factors, i.e., the number of $\bcode{1}$'s in its binary representation: $\deg\gr{j}:=\sum_{\ell=1}^n \bcode{j}_\ell$ (or the cardinality of its set representation). The interactions of $\deg$-$1$ occur when there is just one $\bcode{1}$ in the \jth\ bit string; these are the so-called ``direct effects'' due to each factor acting alone. The $\deg$-$2$ interactions represent relationships between \emph{pairs of factors}, $\deg$-$3$ interactions represent relationships between \emph{three factors}, and so on.

Therefore, the Fourier transform~$\hatbsf$ on~$\Int_2^n$ can be pictured as a geometric complex; a \emph{hypergraph} where the $n$ factors are the nodes, and the \jth\ index with $\deg\gr{j}=k$ represents the $\gr{k-1}$-simplex formed by the participating~$k$ nodes, with~$\abs{\hat{f}_j}$ indicating its strength or size. In particular, a $\deg$-$2$ interaction is an edge between two nodes, a $\deg$-$3$ interaction is the \emph{area inside the triangle} formed by three nodes (not its edges), and a $\deg$-$4$ interaction is the \emph{volume inside the tetrahedron} formed by four nodes (not its edges or faces). If $\hat{f}_j = 0$ for $\deg\gr{j}>2$, then the hypergraph is just a graph.

The index annihilation condition of Theorem~\ref{Thm:mthMoment_a} can be interpreted either in terms of the indices' bit strings or their set representations. For $j_1,j_2\ldots,j_m \in \Int_2^n$, Remark~\ref{Rem:n-tupple_Annihilation} requires the \ellth\ position of their $n$-bit codes to contain an \emph{even number of $\bcode{1}$'s}. Equivalently, in set notation this means that each factor must appear an \emph{even number of times across the~$m$ sets}; then the symmetric difference of the $m$ sets will result in the empty set. For example, suppose we have $n=4$ factors, with vertices labeled as $\ell = 1,2,3,4$. Say we are interested in the $m=3$rd moment. Then for indices $j_1=6$, $j_2=11$, $j_3=13$, we have $j_1 \oplus j_2 \oplus j_3 = 0$ because
$\vect{\bcode{0},\bcode{1},\bcode{1},\bcode{0}} \oplus \vect{\bcode{1},\bcode{1},\bcode{0},\bcode{1}} \oplus \vect{\bcode{1},\bcode{0},\bcode{1},\bcode{1}} = \vect{\bcode{0},\bcode{0},\bcode{0},\bcode{0}}$, as well as
$\set{2,3} \triangle \set{1,2,4} \triangle \set{1,3,4} = \varnothing$,
where $\triangle$ is the \emph{symmetric difference operator}.

\subsection{General form of central moments on~$\Int_2^n$} \label{Sec:CentralMoments_Z2n}
In this section we list the explicit general form for the \mth\ central moments, $m=2\text{--}6$, of a function $\bs{f}:\Int_2^n\rightarrow\C$. Expanding the dot product~\eqref{Eq:mu_m^a=g*gtg*g} with $a=\mu$, like terms are grouped together to obtain the specific form of~\eqref{ThmEq:mthMoment_mu}. For even~$m$ let $p=m/2$, and for odd $m$ let $p=\gr{m+1}/2$ (other combinations of $m,p$ exist, too). Since~$\bs{f}$ is assumed to be complex-valued, $\mu_2 \neq \sigma^2$, so the second moment is included for completeness.

Notice the terms in each \mth\ moment contain exactly~$m$ Fourier coefficients as dictated by Theorem~\ref{Thm:mthMoment_a}. Also, some summations explicitly impose the index annihilation constraint and some do not --- it is not necessary when all coefficients in a term are raised to an even power because the indices self-annihilate.

\begin{itemize}
\item \textbf{Second central moment}
\begin{equation} \label{Eq:mu2_f_Z2n}
\mu_2\gr{\bs{f}} \;=\;
\Dotprod{\dimin{\hatbsf}{\mu}}{\dimin{\hatbsf}{\mu}}
\;=\; \sum_{j \in \:\! \Int_2^n\backslash0} \hat{f}_j^2
\end{equation}

\item \textbf{Third central moment}
\begin{equation} \label{Eq:mu3_f_Z2n}
\mu_3\gr{\bs{f}} \;=\;
\Dotprod{\dimin{\hatbsf}{\mu}}{\conv^2 \dimin{\hatbsf}{\mu}}
\;=\; \;\; 6\!\!\!\! \sum_{\substack{k \oplus l \oplus m = 0\\[2pt] k<l<m \;\!\in\;\! \Int_2^n\backslash0}} \!\!\!\!\!\hat{f}_k \hat{f}_l \hat{f}_m
\end{equation}

\item \textbf{Fourth central moment}
\begin{eqnarray} \label{Eq:mu4_f_Z2n}
\mu_4\gr{\bs{f}} &=&
\Dotprod{\!\conv^2 \!\dimin{\hatbsf}{\mu}}{\conv^2 \dimin{\hatbsf}{\mu}} \nonumber \\[10pt]
&=&
\sum_{j \in \Int_2^n\backslash0} \hat{f}_j^4 \;\;+\;\; 6 \!\!\!\sum_{l < m \in \Int_2^n\backslash0} \!\!\hat{f}_l^2 \hat{f}_m^2
\;\;+\;\; 24\!\!\!\!\!\!\!\sum_{\substack{l \oplus m \oplus p \oplus q = 0 \\[2pt] l<m<p<q \:\!\in \Int_2^n\backslash0}} \!\!\!\!\!\!\! \hat{f}_l \hat{f}_m \hat{f}_p \hat{f}_q \qquad
\end{eqnarray}

\item \textbf{Fifth central moment}
\begin{eqnarray} \label{Eq:mu5_f_Z2n}
\mu_5\gr{\bs{f}} &=&
\Dotprod{\!\conv^2 \!\dimin{\hatbsf}{\mu}}{\conv^3 \dimin{\hatbsf}{\mu}} \nonumber \\[10pt]
&=&
20\!\!\!\!\sum_{\substack{l \oplus m \oplus p \,=\, 0 \\[2pt] \substack{q\,=\,l,m,p \\[2pt] l<m<p \:\!\in \Int_2^n\backslash0}}} \!\!\!\! \hat{f}_l \hat{f}_m \hat{f}_p \hat{f}_q^2
        \;\;\;+\;\;\;\,
60\!\!\!\!\!\sum_{\substack{l \oplus m \oplus p \,=\, 0 \\[2pt] \substack{q\,\neq\,l,m,p \:\!\in \Int_2^n\backslash0\\[2pt] l<m<p \:\!\in \Int_2^n\backslash0}}} \!\!\!\!\! \hat{f}_l \hat{f}_m \hat{f}_p \hat{f}_q^2 \nonumber \\[10pt]
&&
\qquad\qquad + \;\;\; 120\!\!\!\!\!\sum_{\substack{l \oplus m \oplus p \oplus q \oplus r= 0 \\[2pt] l<m<p<q<r \:\!\in \Int_2^n\backslash0}} \!\!\!\!\! \hat{f}_l \hat{f}_m \hat{f}_p \hat{f}_q \hat{f}_r
\end{eqnarray}

\item \textbf{Sixth central moment}
\begin{eqnarray} \label{Eq:mu6_f_Z2n}
\mu_6\gr{\bs{f}} &=&
\Dotprod{\!\conv^3 \!\dimin{\hatbsf}{\mu}}{\conv^3 \dimin{\hatbsf}{\mu}} \nonumber \\[10pt]
&=&
\sum_{j \in \Int_2^n\backslash0} \!\! \hat{f}_j^6
        \;\;+\;\; 15 \!\!\!\sum_{\substack{p = l,m\\[2pt] l<m \in \Int_2^n\backslash0}} \!\! \hat{f}_l^2 \hat{f}_m^2 \hat{f}_p^2
        \;\;+\;\; 90 \!\!\!\sum_{\substack{p \neq l,m \in \Int_2^n\backslash0\\[2pt] l<m \in \Int_2^n\backslash0}} \!\! \hat{f}_l^2 \hat{f}_m^2 \hat{f}_p^2 \qquad\qquad \nonumber \\[10pt]
&&
+\;\;\; 120\!\!\!\!\!\sum_{\substack{l \oplus m \oplus p \oplus q = 0 \\[2pt] \substack{r\,=\,l,m,p,q \\[2pt] \;l<m<p<q \:\!\in \Int_2^n\backslash0}}} \!\!\!\!\! \hat{f}_l \hat{f}_m \hat{f}_p \hat{f}_q \hat{f}_r^2
\;\;+\;\;
360\!\!\!\!\sum_{\substack{l \oplus m \oplus p \oplus q \,=\, 0 \\[2pt] \substack{r\,\neq\,l,m,p,q \:\!\in \Int_2^n\backslash0\\[2pt] l<m<p<q \:\!\in \Int_2^n\backslash0}}} \!\!\!\! \hat{f}_l \hat{f}_m \hat{f}_p \hat{f}_q \hat{f}_r^2 \qquad\quad \nonumber \\[10pt]
&&
\qquad\qquad + \;\;\; 720\!\!\!\!\!\!\sum_{\substack{l \oplus m \oplus p \oplus q \oplus r \oplus s = 0 \\[2pt] l<m<p<q<r<s \:\!\in \Int_2^n\backslash0}} \!\!\!\!\!\! \hat{f}_l \hat{f}_m \hat{f}_p \hat{f}_q \hat{f}_r \hat{f}_s \qquad\qquad
\end{eqnarray}
\end{itemize}

\begin{remark}
It is curious to note the recursive nature of the cental moments in~{\eqref{Eq:mu2_f_Z2n}--\eqref{Eq:mu6_f_Z2n}}: for $m\geq2$, $\mu_m$ inherits all combinations of the terms from~$\mu_{m-2}$ multiplied by a factor of~$\hat{f}_{q}^2$ for some $q\in\Int_2^n$, as well as the addition of unique linear terms $\hat{f}_{j_1} \hat{f}_{j_2} \cdots \hat{f}_{j_m}$ (i.e., $j_1 \neq j_2 \neq \cdots \neq j_m$).

For example, $\mu_0=1$ (see~\eqref{Eq:mthMoment_base_cases}), so all the terms of $\mu_2$ in~\eqref{Eq:mu2_f_Z2n} are simply $1\cdot\hat{f}_{q}^2 = \hat{f}_{q}^2$. Next, $\mu_1=0$, so only the linear terms $\hat{f}_{j_1} \hat{f}_{j_2} \hat{f}_{j_3}$ remain for~$\mu_3$ in~\eqref{Eq:mu3_f_Z2n}. More interesting are the cental moments for $m\geq4$. For instance, $\mu_2$ has terms of the form $\hat{f}_{p}^2$, so the terms of $\mu_4$ in~\eqref{Eq:mu4_f_Z2n} are $\hat{f}_{p}^2 \cdot \hat{f}_{q}^2$ for both $q=p$ and $q \neq p$, as well as linear terms $\hat{f}_{j_1} \hat{f}_{j_2} \hat{f}_{j_3} \hat{f}_{j_4}$. We encourage the reader to verify that $\mu_5$ in~\eqref{Eq:mu5_f_Z2n} and $\mu_6$ in~\eqref{Eq:mu6_f_Z2n} follow in the same way.
\end{remark}

\subsection{Statistics of binomial distributions from the Fourier transform} \label{Sec:BinomialDist_Stats}
Given that we are dealing with binary variables, it should not come as too much of a surprise that we can derive the statistics of a Bernoulli process~$\bs{f}:\Int_2^n\rightarrow\Real$ from its Fourier transform and without its probability mass function. The statistics for a binomial distribution with probability of success~$p$ and failure~$q$ are well known~\cite{skorski2025handyBinomialFormulas}.
As a reference, the first central moments when $p=q=1/2$ are listed in Table~\ref{Tab:CentralMoments BinomDist}.

\begin{table}[!htb]
\centering
  \caption{Central moments of a binomial distribution when $p=q=\frac{1}{2}$} \label{Tab:CentralMoments BinomDist}
  \vspace{3pt}
  \begin{tabular}{|lrcl|}
  \hline
  & & & \\[-9pt]
  & $\sigma^2$ & $=$ & $\displaystyle npq\big|_{p=q=\frac{1}{2}} \,\;=\;\; \frac{n}{4}$ \\[10pt]
  & $\mu_3$ & $=$ & $\displaystyle npq\gr{q-p}\big|_{p=q=\frac{1}{2}} \,\;=\;\; 0$ \\[7pt]
  & $\mu_4$ & $=$ & $\displaystyle 3\gr{npq}^2 + n\Gr{pq - 6\gr{pq}^2}\big|_{p=q=\frac{1}{2}} \,\;=\;\; \frac{3n^2-2n}{16}$ \\[10pt]
  & $\mu_5$ & $=$ & $\displaystyle \Gr{10\gr{npq}^2 + n\Gr{pq - 12\gr{pq}^2}}\gr{q-p}\big|_{p=q=\frac{1}{2}} \,\;=\;\; 0$ \\[10pt]
  & $\mu_6$ & $=$ & $\displaystyle 15\gr{npq}^3 + n^2\Gr{25\gr{pq}^2 - 130\gr{pq}^3} + n\Gr{120\gr{pq}^3 - 30\gr{pq}^2 + pq}\big|_{p=q=\frac{1}{2}}$\; \\[7pt]
  & & $=$ & $\displaystyle \frac{15n^3 - 30n^2 + 16n}{64}$ \\[9pt]
\hline
\end{tabular}
\end{table}

Our goal is to formulate these statistics using our Fourier domain method. To make the setup apt, we impose the condition that the~$n$ binary variables which affect function~$\bs{f}$ are independent, i.e., there are \emph{no interactions}. If there are no interactions between variables, then the only nonzero coefficients of~$\hatbsf$ are the \emph{direct effects} of each variable. Clearly these effects should all be equal (none are more important than any other), so we set all $\deg$-$1$ coefficients to have the same value $d\in\Real$:
\begin{equation} \label{Eq:DirectEffect_fhat=d}
\qquad
\hat{f}_j \;=\;
\left\{
  \begin{array}{ll}
    d, & \hbox{if $j=2^{\ell-1} = \vect{\,\underbrace{\bcode{0},\ldots,\bcode{0},\bcode{1},\bcode{0},\ldots\bcode{0}}_{\text{only the $\ell$th bit = \bcode{1}}}\,}$, for $\ell=1,2,\ldots,n$;} \\[-3pt]
    0, & \hbox{otherwise.}
  \end{array}
\right.
\end{equation}

For the variance~\eqref{Eq:variance_|fj|^2} (or~\eqref{Eq:mu2_f_Z2n} since~$\bs{f}$ is assumed real), if the only nonzero coefficients of~$\hatbsf$ are the~$n$ direct effects~\eqref{Eq:DirectEffect_fhat=d}, then
\begin{equation} \label{Eq:var_binom}
\sigma^2 \,=\, nd^2.
\end{equation}
Recall the other central moments in Section~\ref{Sec:CentralMoments_Z2n}. Examining the third moment~\eqref{Eq:mu3_f_Z2n}, it is immediately clear that
\begin{equation} \label{Eq:mu3_binom}
\mu_3 \,=\, 0
\end{equation}
because there are no $\hat{f}_k \hat{f}_l \hat{f}_m$ terms that satisfy $k \oplus l \oplus m = 0$ (i.e., the single $\bcode{1}$ bit in each $k < l < m$ index will be in different positions and thus will not annihilate).

For the same reason, there cannot be any $\hat{f}_l \hat{f}_m \hat{f}_p \hat{f}_q$ terms in the fourth moment~\eqref{Eq:mu4_f_Z2n}. So, $\mu_4 = \sum_j \hat{f}_j^4 + 6 \sum_{l<m} \hat{f}_l^2 \hat{f}_m^2 = \AutoGroup{\sum_j \hat{f}_j^2}^2 + 4 \sum_{l<m} \hat{f}_l^2 \hat{f}_m^2$. Substituting~\eqref{Eq:DirectEffect_fhat=d}, we get
\begin{equation} \label{Eq:mu4_binom}
\mu_4 \,=\, \Gr{n d^2}^2 + 4\displaystyle\binom{n}{2}d^2d^2 \,=\, \gr{3n^2 - 2n} \:\! d^4.
\end{equation}

Just like the third moment, the fifth moment~\eqref{Eq:mu5_f_Z2n} is
\begin{equation} \label{Eq:mu5_binom}
\mu_5 \,=\, 0
\end{equation}
because there are no $\hat{f}_l \hat{f}_m \hat{f}_p \hat{f}_q^2$ or $\hat{f}_l \hat{f}_m \hat{f}_p \hat{f}_q \hat{f}_r$ terms that satisfy the annihilation condition.\footnote{This can be generalized beyond $\deg$-$1$ coefficients for odd~$m$: there is no combination of Fourier coefficients of any odd degree such that the indices of $\hat{f}_{j_1} \hat{f}_{j_2} \cdots \hat{f}_{j_m}$ will annihilate. Thus, if a function defined on~$\Int_2^n$ only has odd-degree Fourier coefficients, then all of its odd-order moments will be zero.}

The sixth moment~\eqref{Eq:mu6_f_Z2n} follows in a similar way to the fourth (there are no $\hat{f}_l \hat{f}_m \hat{f}_p \hat{f}_q \hat{f}_r^2$ or $\hat{f}_l \hat{f}_m \hat{f}_p \hat{f}_q \hat{f}_r \hat{f}_s$ terms), albeit with a bit more effort, to yield
\begin{eqnarray} \label{Eq:mu6_binom}
\mu_6 &=& \sum_j \hat{f}_j^6 + 15 \sum_{p=l<m} \hat{f}_l^2 \hat{f}_m^2 \hat{f}_p^2 + 90 \sum_{p\neq l<m} \hat{f}_l^2 \hat{f}_m^2 \hat{f}_p^2 \nonumber \\[2pt]
&=& \GR{\sum_j \hat{f}_j^2}^3 + 12 \sum_{p=l<m} \hat{f}_l^2 \hat{f}_m^2 \hat{f}_p^2 + 84 \sum_{p\neq l<m} \hat{f}_l^2 \hat{f}_m^2 \hat{f}_p^2 \nonumber \\[2pt]
&=& \Gr{nd^2}^3 + 24\binom{n}{2}d^2d^2d^2 + 84\binom{n}{3}d^2d^2d^2 \nonumber \\[5pt]
&=& \gr{15n^3-30n^2+16n} \:\! d^6.
\end{eqnarray}

Comparing the central moments \eqref{Eq:var_binom}--\eqref{Eq:mu6_binom} derived from the Fourier domain with the classically derived ones for a binomial distribution with $p=q=1/2$ in Table~\ref{Tab:CentralMoments BinomDist}, we observe that they are the same if we set $d=1/2$. Further, these central moments can be used to find the skewness, kurtosis, hyperskewness, and hyperkurtosis associated with this function for any direct effect~$d\in\Real$ (not just $d=1/2$); notice in Table~\ref{Tab:StandMoments BinomDist} that the standardized moments do not depend on~$d$.

\begin{table}[!htb]
\centering
\caption{Standardized moments for $n$ indep.~binary factors with equal effect~$d$} \label{Tab:StandMoments BinomDist}
  \vspace{3pt}
  \begin{tabular}{|crlc|}
  \hline
  & & & \\[-10pt]
  & \underline{Skewness}      & $\displaystyle \gamma \,:=\, \frac{\mu_3}{\sigma^3} \,=\, \frac{0}{\gr{nd^2}^{3/2}} \,=\, 0$ & \\[12pt]
  & \underline{Kurtosis}      & $\displaystyle \kappa \,:=\, \frac{\mu_4}{\sigma^4} \,=\, \frac{\gr{3n^2-2n} \:\! d^4}{\gr{nd^2}^{2}} \,=\, 3 - \frac{2}{n}$ & \\[12pt]
  & \underline{Hyperskewness} & $\displaystyle \gamma_5 \,:=\, \frac{\mu_5}{\sigma^5} \,=\, \frac{0}{\gr{nd^2}^{5/2}} \,=\, 0$ & \\[12pt]
  & \underline{Hyperkurtosis} & $\displaystyle \kappa_6 \,:=\, \frac{\mu_6}{\sigma^6} \,=\, \frac{\gr{15n^3-30n^2+16n} \:\! d^6}{\gr{nd^2}^{3}} \,=\, 15 - \frac{30}{n} + \frac{16}{n^2}$ & \\[12pt]
\hline
\end{tabular}
\end{table}

\section{Applications} \label{Sec:Applications}
In this section, we present four examples to highlight the utility of Theorem~\ref{Thm:mthMoment_a}. Section~\ref{Sec:1D_DFT_example} contains a $1$-dimensional DFT both to affirm that a function's statistics can be \emph{calculated} by its Fourier transform, and to demonstrate a non-Gaussian distribution resulting from a single variable. Section~\ref{Sec:Genetics_DFT_example} contains a genetics example, where we \emph{analyze} the skewness of a trait via the hypergraph of interactions due to $n=13$ different amino acid loci. Section~\ref{Sec:Design_DFT_example} shows how to \emph{design} a nonlinear process from~$n$ binary variables. And Section~\ref{Sec:FeasibilityConstraint_DFT_example} provides a sketch employing the annihilation property as a \emph{constraint} for phase retrieval.

\subsection{One-dimensional example: Canonical DFT} \label{Sec:1D_DFT_example}
This first example simply shows how the moments of a function can be computed from its Fourier transform. Consider an arbitrary function $\bs{f}:\Int_{64}\rightarrow\Real$ whose Fourier transform~$\hatbsf$ is sparse; Table~\ref{Tab:DFT_Z64_example_8Fourier_coef} lists such a transform with eight nonzero elements that were randomly generated. Notice that since~$\bs{f}$ is real-valued, its transform is conjugate symmetric: ${\hatbsf}_{\!\ominus} = \overline{\hatbsf\,}$.

\begin{table}[!h]
\centering
  \caption{The eight nonzero coefficients of~$\hatbsf$} \label{Tab:DFT_Z64_example_8Fourier_coef}
  \vspace{5pt}
  \begin{tabular}{c|rccc|r}
      \multicolumn{1}{c}{$j$} & \multicolumn{1}{c}{$\hat{f}_j$} &&&
      \multicolumn{1}{c}{$j$} & \multicolumn{1}{c}{$\hat{f}_j$} \\
      \cline{1-2} \cline{5-6}
      $3$  & $ 1.22 + 0.19\imag$ &&& $54$ & $ 0.12 + 0.96\imag$ \\
      $4$  & $-0.39 - 1.15\imag$ &&& $58$ & $-0.69 + 0.24\imag$ \\
      $6$  & $-0.69 - 0.24\imag$ &&& $60$ & $-0.39 + 1.15\imag$ \\
      $10$ & $ 0.12 - 0.96\imag$ &&& $61$ & $ 1.22 - 0.19\imag$ \\
      \cline{1-2} \cline{5-6}
  \end{tabular}
\end{table}

\begin{figure}[!htb]
\centering
\includegraphics[scale=0.36, trim={160 28 25 -10mm},clip] 
{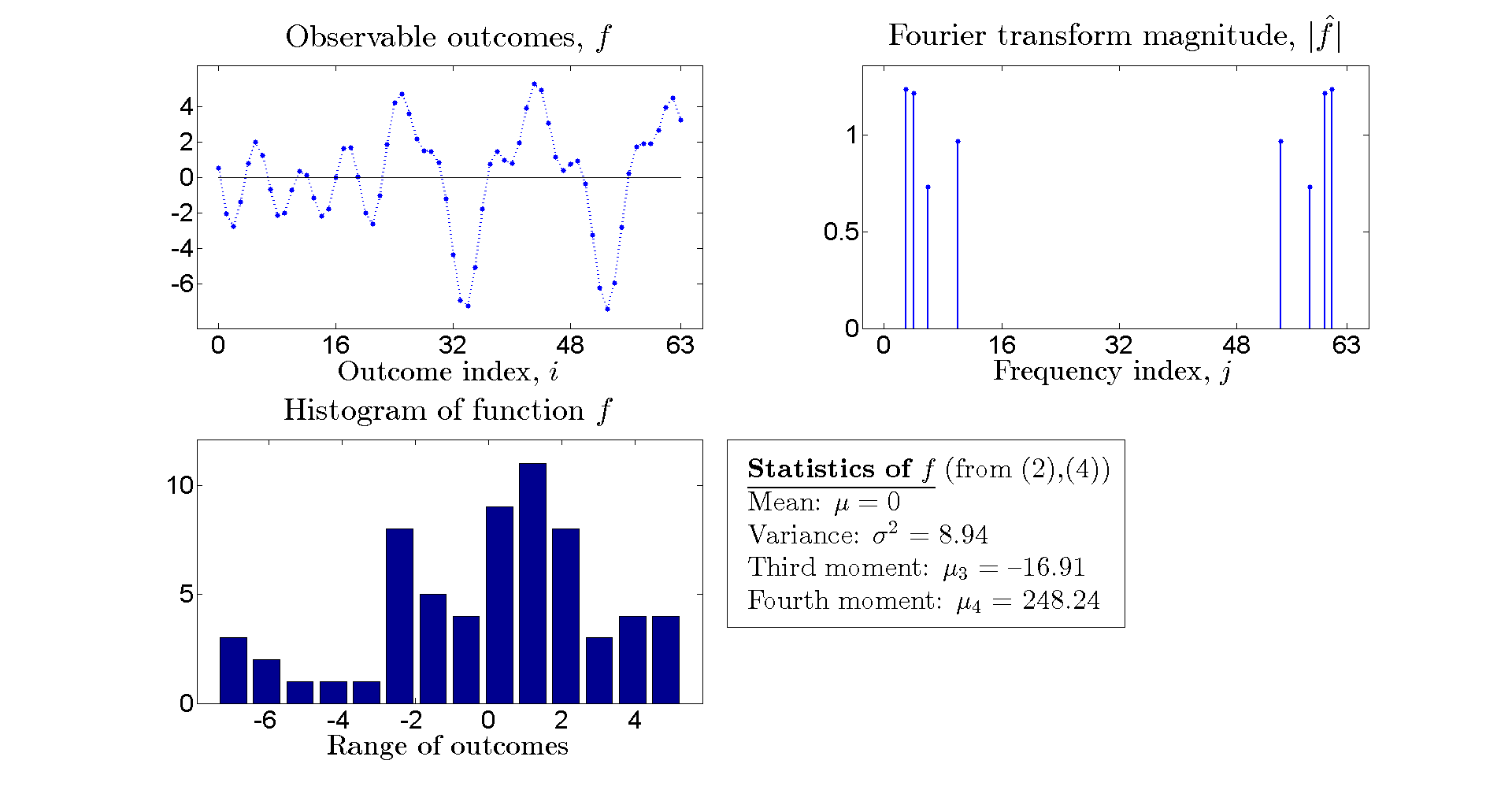}
\caption{A synthetically-generated $1$-dim example defined on $\Int_{64}$. (\textbf{Top}) The function~$\bs{f}$ and its sparse Fourier transform~$\hatbsf$. (\textbf{Bottom}) The histogram of function~$\bs{f}$ and table of its statistics. \label{Fig:DFT_Z64_example}}
\end{figure}

We do not list the $64$ elements of~$\bs{f}$, but they are easily obtained by taking the inverse Fourier transform~\eqref{Def:inverseDFT} of~$\hatbsf$. Figure~\ref{Fig:DFT_Z64_example} shows function~$\bs{f}$, the histogram of its values, and the magnitude of its Fourier transform. Also shown is a table of some summary statistics of~$\bs{f}$ calculated the \emph{traditional way} using~\eqref{Def:mthRawCentralMoments} and~\eqref{Def:Variance}.

Now we confirm that these statistics of~$\bs{f}$ can also be calculated solely from the coefficients of~$\hatbsf$ in Table~\ref{Tab:DFT_Z64_example_8Fourier_coef}. Immediately, we notice that $\hat{f}_0$ is not listed in Table~\ref{Tab:DFT_Z64_example_8Fourier_coef}, so $\hat{f}_0 = 0$, which means $\mu = 0$ from~\eqref{Eq:hat{f}=mu}. Next, the variance from~\eqref{Eq:variance_|fj|^2} is simply
$$
\sigma^2 \;=\; \abs{\hat{f}_3}^2 + \abs{\hat{f}_4}^2 + \abs{\hat{f}_6}^2 + \abs{\hat{f}_{10}}^2 + \abs{\hat{f}_{54}}^2 + \abs{\hat{f}_{58}}^2 + \abs{\hat{f}_{60}}^2 + \abs{\hat{f}_{61}}^2.
$$
Using~\eqref{Eq:mu_m^a=g*gtg*g} with $a=\mu$, the third and fourth central moments of~$\bs{f}$ are
$$
\mu_3 \;=\; 3 \:\!\Gr{\hat{f}_3^2 \hat{f}_{58} + \hat{f}_6 \hat{f}_{61}^2} + 6 \Gr{\hat{f}_4 \hat{f}_6 \hat{f}_{54} + \hat{f}_{10} \hat{f}_{58} \hat{f}_{60}}
$$
and
\begin{eqnarray*}
\mu_4 &=&
6 \Gr{\hat{f}_3^2 \hat{f}_{61}^2 + \hat{f}_4^2 \hat{f}_{60}^2 + \hat{f}_6^2 \hat{f}_{58}^2 + \hat{f}_{10}^2 \hat{f}_{54}^2} +
12 \Gr{\hat{f}_3^2 \hat{f}_4 \hat{f}_{54} + \hat{f}_{10} \hat{f}_{60} \hat{f}_{61}^2} \\[5pt]
&& \quad +\: 24 \GR{\hat{f}_3 \hat{f}_4 \hat{f}_{60} \hat{f}_{61} + \hat{f}_3 \hat{f}_6 \hat{f}_{58} \hat{f}_{61} +  \hat{f}_3 \hat{f}_{10} \hat{f}_{54} \hat{f}_{61} + \hat{f}_4 \hat{f}_6 \hat{f}_{58} \hat{f}_{60} \\[3pt]
&& \qquad\qquad\qquad +\: \hat{f}_4 \hat{f}_{10} \hat{f}_{54} \hat{f}_{60} + \hat{f}_6 \hat{f}_{10} \hat{f}_{54} \hat{f}_{58}}.
\end{eqnarray*}
Substituting the eight Fourier coefficients from Table~\ref{Tab:DFT_Z64_example_8Fourier_coef} into these expressions, we get
$$
\sigma^2 = 8.94, \qquad\quad \mu_3 = -16.91, \qquad\quad \mu_4 = 248.24
$$
which agree with the ``Statistics of~$\bs{f}$'' table in Figure~\ref{Fig:DFT_Z64_example}. The calculations give identical results whether done traditionally or in the Fourier domain, but it is clearly simpler in the latter because~$\hatbsf$ is sparse. Further, just as Theorem~\ref{Thm:mthMoment_a} claims, all of the terms in~$\mu_3$ and~$\mu_4$ contain precisely three and four coefficients, respectively. The reader is encouraged to confirm that the indices in each of these terms add up to ${0\!\pmod{64}}$, as required.

Let us calculate the skewness and kurtosis from the third and fourth moments:
$$
\gamma \,:=\, \frac{\mu_3}{\sigma^3} \,=\, \frac{-16.91}{8.94^{3/2}} \,=\, -0.63, \quad\qquad
\kappa \,:=\, \frac{\mu_4}{\sigma^4} \,=\, \frac{248.24}{8.94^2} \,=\, 3.11.
$$
Compared to the canonical normal distribution ($\gamma_\text{normal} = 0$, $\kappa_{\text{normal}} = 3$), these values are not extreme deviations, but they are also not trivial. This function is composed of just four cosines (recall,~$\hatbsf$ is conjugate symmetric), each with skewness $\gamma_{\text{cos}}=0$ and platykurtosis $\kappa_{\text{cos}}=1.5$. It is notable that their linear combination results in skewness and kurtosis that are markedly different.

\subsection{Multidimensional example: Genetics} \label{Sec:Genetics_DFT_example}
Consider the \emph{Entacmaea quadricolor} anemone sea creature, which has a protein that determines whether it fluoresces red or blue~\cite{doro2021fourierTransQuantTraitCS, LearningPatternEpistasis_Poelwijk2019}. The protein has $n=13$ loci, each occupied by one of two different amino acids, that are the influencing ``genes''; so there are $2^{13}=8192$ genotypes. \emph{Important note:} in this example we list the factors \emph{right to left} as is done in~\cite{doro2021fourierTransQuantTraitCS} (however,~\cite{LearningPatternEpistasis_Poelwijk2019} lists left to right). Thus, the \ith\ genotype is $i=\vect{\bcode{i}_{13},\ldots,\bcode{i}_2,\bcode{i}_1}$, where $\bcode{i}_\ell\in\Int_2$ represents which amino acid is present at locus $\ell$.

Let the trait~$\bs{f}:\Int_2^{13}\rightarrow\Real$ represent the measured brightness of the \emph{Entacmaea quadricolor,} which relates to its color. That is, the \ith\ organism fluoresces with brightness $f_i\geq0$ as a function of the \ith\ genotype. The left side of Figure~\ref{Fig:DFT_Z2n_example} shows the full trait over all possible genotypes and its histogram of outcomes. The data appears to be trimodal with a very strong skew compared to a normal distribution.
\begin{figure}[!htb]
\centering
\includegraphics[scale=0.36, trim={128 27 25 0mm},clip] 
{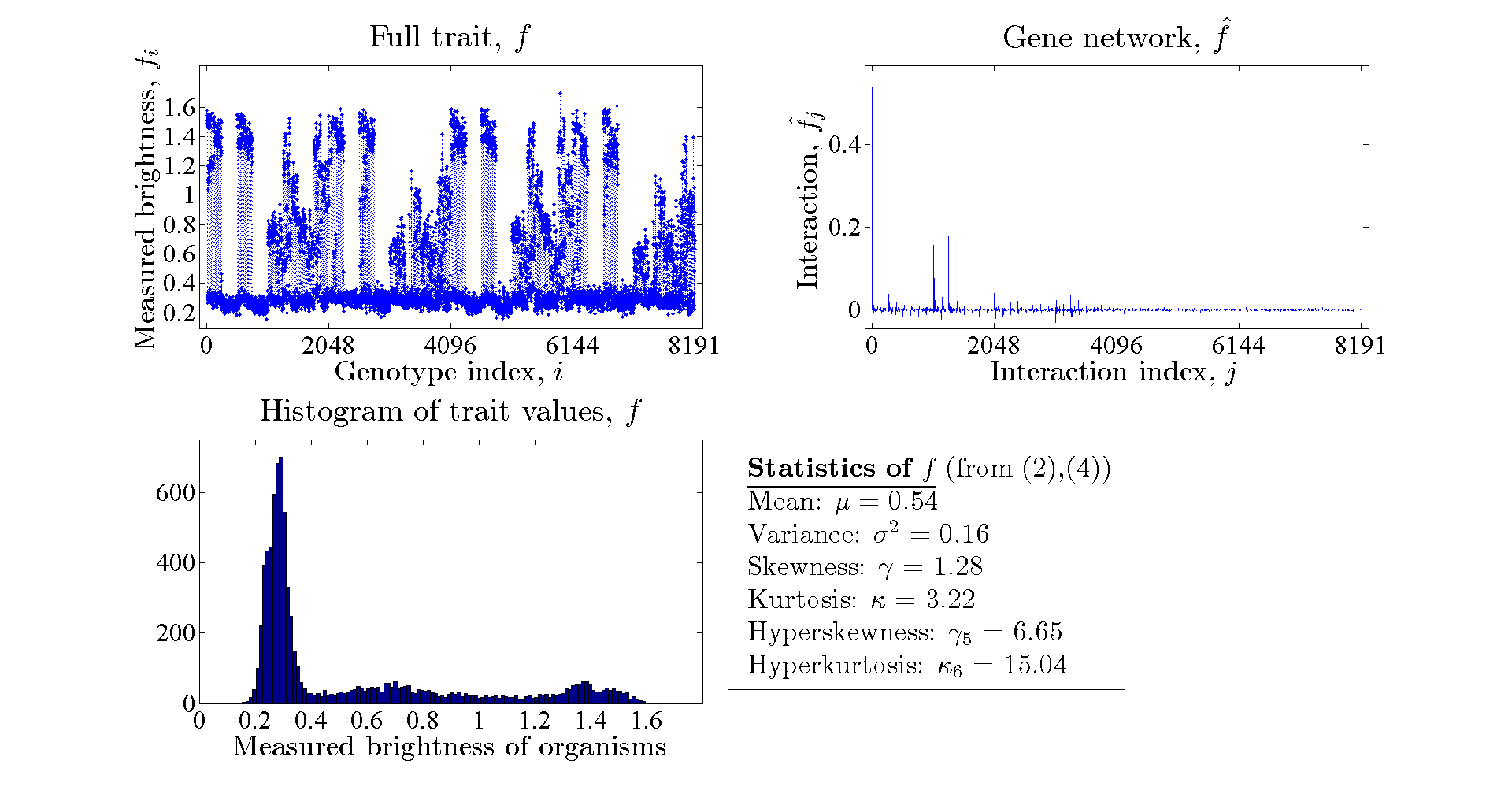}
\caption{(\textbf{Top}) The full trait~$\bs{f}$ and its sparse gene network~$\hatbsf$. (\textbf{Bottom}) The histogram of trait~$\bs{f}$ and table of its statistics. \label{Fig:DFT_Z2n_example}}
\end{figure}

In the context of genetics, the Fourier transform~$\hatbsf$ is sometimes called a ``gene network''~\cite{doro2021fourierTransQuantTraitCS}. The top-right of Figure~\ref{Fig:DFT_Z2n_example} shows that the gene network for this trait is fairly sparse. The top ten genetic interactions are listed in Table~\ref{Tab:Entacmaea_top-10_interactions}. The first three columns are \emph{equivalent representations} of index $j\in\Int_2^{13}$: \emph{decimal, binary, set}. The set notation simply lists which loci are present in a given interaction. Notice how these most important interactions are all low-degree\footnote{Degree is referred to as ``level'' in~\cite{doro2021fourierTransQuantTraitCS}.}. The zeroth interaction is the average trait value $\mu=0.5381$ (see~\eqref{Eq:hat{f}=mu}), which makes sense since the bulk of the trait values are concentrated around $0.3$.

\begin{table}[!hbt]
\centering
\setlength\tabcolsep{2.5pt}
\caption{The top ten coefficients of gene network $\hatbsf$} \label{Tab:Entacmaea_top-10_interactions}
\vspace{5pt}
\begin{tabular}{c|ccccccccccccc|c|c||c}

    \multicolumn{1}{c}{$j$} &
    $\,[\bcode{j}_{13}\!$ & $\bcode{j}_{12}$ & $\bcode{j}_{11}$ & $\bcode{j}_{10}$ & $\bcode{j}_{9\;}$ & $\bcode{j}_{8\;}$ & $\bcode{j}_{7\;}$ & $\bcode{j}_{6\;}$ & $\bcode{j}_{5\;}$ & $\bcode{j}_{4\;}$ & $\bcode{j}_{3\;}$ & $\bcode{j}_{2\;}$ & \multicolumn{1}{c}{$\bcode{j}_1]\,$}
    & \multicolumn{1}{c}{$\text{Loci}_j$} & \multicolumn{1}{c}{$\text{deg}\gr{j}$} & \multicolumn{1}{c}{$\quad\: \hat{f}_j \quad\:$} \\

    \hline

    $0\,$ & $\![\,\bcode{0}$ & $\bcode{0}$ & $\bcode{0}$ & $\bcode{0}$ & $\bcode{0}$ & $\bcode{0}$ & $\bcode{0}$ & $\bcode{0}$ & $\bcode{0}$ & $\bcode{0}$ & $\bcode{0}$ & $\bcode{0}$ & $\,\bcode{0}\,]$ & $\varnothing\,$ & $0$ & $0.5381$ \\

    $264\,$ & $\![\,\bcode{0}$ & $\bcode{0}$ & $\bcode{0}$ & $\bcode{0}$ & $\bcode{1}$ & $\bcode{0}$ & $\bcode{0}$ & $\bcode{0}$ & $\bcode{0}$ & $\bcode{1}$ & $\bcode{0}$ & $\bcode{0}$ & $\,\bcode{0}\,]$ & $\set{4,9}\,$ & $2$ & $0.2396$ \\

    $1280\,$ & $\![\,\bcode{0}$ & $\bcode{0}$ & $\bcode{1}$ & $\bcode{0}$ & $\bcode{1}$ & $\bcode{0}$ & $\bcode{0}$ & $\bcode{0}$ & $\bcode{0}$ & $\bcode{0}$ & $\bcode{0}$ & $\bcode{0}$ & $\,\bcode{0}\,]$ & $\set{9,11}\,$ & $2$ & $0.1778$ \\

    $1032\,$ & $\![\,\bcode{0}$ & $\bcode{0}$ & $\bcode{1}$ & $\bcode{0}$ & $\bcode{0}$ & $\bcode{0}$ & $\bcode{0}$ & $\bcode{0}$ & $\bcode{0}$ & $\bcode{1}$ & $\bcode{0}$ & $\bcode{0}$ & $\,\bcode{0}\,]$ & $\set{4,11}\,$ & $2$ & $0.1565$ \\

    $8\,$ & $\![\,\bcode{0}$ & $\bcode{0}$ & $\bcode{0}$ & $\bcode{0}$ & $\bcode{0}$ & $\bcode{0}$ & $\bcode{0}$ & $\bcode{0}$ & $\bcode{0}$ & $\bcode{1}$ & $\bcode{0}$ & $\bcode{0}$ & $\,\bcode{0}\,]$ & $\set{4}\,$ & $1$ & $0.1019$ \\

    $256\,$ & $\![\,\bcode{0}$ & $\bcode{0}$ & $\bcode{0}$ & $\bcode{0}$ & $\bcode{1}$ & $\bcode{0}$ & $\bcode{0}$ & $\bcode{0}$ & $\bcode{0}$ & $\bcode{0}$ & $\bcode{0}$ & $\bcode{0}$ & $\,\bcode{0}\,]$ & $\set{9}\,$ & $1$ & $0.0934$ \\

    $2048\,$ & $\![\,\bcode{0}$ & $\bcode{1}$ & $\bcode{0}$ & $\bcode{0}$ & $\bcode{0}$ & $\bcode{0}$ & $\bcode{0}$ & $\bcode{0}$ & $\bcode{0}$ & $\bcode{0}$ & $\bcode{0}$ & $\bcode{0}$ & $\,\bcode{0}\,]$ & $\set{12}\,$ & $1$ & $0.0401$ \\

    $280\,$ & $\![\,\bcode{0}$ & $\bcode{0}$ & $\bcode{0}$ & $\bcode{0}$ & $\bcode{1}$ & $\bcode{0}$ & $\bcode{0}$ & $\bcode{0}$ & $\bcode{1}$ & $\bcode{1}$ & $\bcode{0}$ & $\bcode{0}$ & $\,\bcode{0}\,]$ & $\set{4,5,9}\,$ & $3$ & $0.0380$ \\

    $2312\,$ & $\![\,\bcode{0}$ & $\bcode{1}$ & $\bcode{0}$ & $\bcode{0}$ & $\bcode{1}$ & $\bcode{0}$ & $\bcode{0}$ & $\bcode{0}$ & $\bcode{0}$ & $\bcode{1}$ & $\bcode{0}$ & $\bcode{0}$ & $\,\bcode{0}\,]$ & $\set{4,9,12}\,$ & $3$ & $0.0360$ \\

    $\,3328\,$ & $\![\,\bcode{0}$ & $\bcode{1}$ & $\bcode{1}$ & $\bcode{0}$ & $\bcode{1}$ & $\bcode{0}$ & $\bcode{0}$ & $\bcode{0}$ & $\bcode{0}$ & $\bcode{0}$ & $\bcode{0}$ & $\bcode{0}$ & $\,\bcode{0}\,]$ & $\,\set{9,11,12}\,$ & $3$ & $0.0352$
    \\
    \hline
\end{tabular}
\end{table}

Viewing the interactions of Table~\ref{Tab:Entacmaea_top-10_interactions}
as $k$-simplices, we see that all of them, except~$\hat{f}_{280}$ (or~$\hat{f}_{\set{4,5,9}}$), are contained within a tetrahedron with vertices labeled $\set{4,9,11,12}$, as seen in Figure~\ref{Fig:Tetrahedron}. Let us investigate the strongly lopsided distribution of the trait values in Figure~\ref{Fig:DFT_Z2n_example} with respect to this tetrahedron. Using~\eqref{Def:mthRawCentralMoments} and~\eqref{Def:Variance}, the skewness and hyperskewness are $\gamma=1.28$ and $\gamma_5=6.65$, which are quite large compared to the symmetric normal distribution ($\gamma_\text{normal} = 0$, $\gamma_{5,\text{normal}} = 0$). Focusing just on the skewness, the third central moment is $\mu_3 = \gamma\sigma^3 = 0.0814$. From~\eqref{Eq:mu3_f_Z2n}, the top ten interactions in Table~\ref{Tab:Entacmaea_top-10_interactions} contribute the following $3$-coefficient terms to~$\mu_3$:
\begin{eqnarray}
& 6\:\!\Gr{ \hat{f}_{\set{4,9}}\hat{f}_{\set{4,11}}\hat{f}_{\set{9,11}} + \hat{f}_{\set{4}}\hat{f}_{\set{9}}\hat{f}_{\set{4,9}} + \hat{f}_{\set{4,9}}\hat{f}_{\set{12}}\hat{f}_{\set{4,9,12}} & \nonumber \\[3pt]
& +\: \hat{f}_{\set{9,11}}\hat{f}_{\set{12}}\hat{f}_{\set{9,11,12}} + \hat{f}_{\set{4,11}}\hat{f}_{\set{4,9,12}}\hat{f}_{\set{9,11,12}}}. & \label{Eq:First5terms_mu3_geneticsExample}
\end{eqnarray}
Notice, each gene locus appears an even number of times within each term, so
the index annihilation condition for set notation is satisfied (e.g., $\set{4,9} \triangle \set{4,11} \triangle \set{9,11} = \varnothing$). Also notice that coefficient~$\hat{f}_{\set{4,5,9}}$ does not appear in any term since there is no way to annihilate locus~$5$ using the other coefficients in Table~\ref{Tab:Entacmaea_top-10_interactions}.

\begin{figure}[!ht]
\centering
\frame{\includegraphics[scale=0.3, trim={0mm 0mm 0mm 0mm},clip] 
{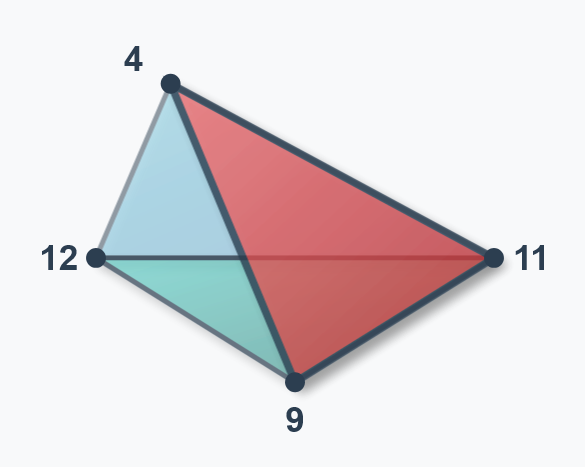}}
\caption{The skew of the trait is primarily affected by the tetrahedron with vertices $\set{4,9,11,12}$. All of the $k$-simplices listed in Table~\ref{Tab:mu3_terms} are part of this tetrahedron. \label{Fig:Tetrahedron}}
\end{figure}

The five terms of~\eqref{Eq:First5terms_mu3_geneticsExample} are listed in the rows of Table~\ref{Tab:mu3_terms}, where the first three columns contain the specific loci of each~$\hat{f}_j$. Substituting the~$\hat{f}_j$ values from Table~\ref{Tab:Entacmaea_top-10_interactions}, the absolute and relative contribution of each term to~$\mu_3$ are listed in the last two columns. The first term consists of the tetrahedron's three edges $\set{4,9}$,$\set{4,11}$,$\set{9,11}$ in Figure~\ref{Fig:Tetrahedron}, and contributes almost $50\%$ to the total skew. The next term is a ``barbell'': two vertices $\set{4}$, $\set{9}$ connected by the edge $\set{4,9}$, constituting about $17\%$ of the skew. The next two terms are composed of a face of the tetrahedron combined with a vertex and an edge opposite to it. Finally, the fifth term consists of two faces $\set{4,9,12}$, $\set{9,11,12}$ connected by the edge $\set{4,11}$, though it only accounts for about $1.5\%$ of the skew.

\begin{table}[!htb]
\centering
  \caption{Contribution of the first five terms of~$\mu_3$ in~\eqref{Eq:First5terms_mu3_geneticsExample}
\label{Tab:mu3_terms}}
  \vspace{5pt}
  \begin{tabular}{c|c|c||c|c}
    \multicolumn{3}{c}{\underline{Loci}} & \multicolumn{1}{c}{\underline{Abs.~contrib.}} & \multicolumn{1}{c}{\underline{Rel.~contrib.}}\\
    \multicolumn{1}{c}{$j_1$} & \multicolumn{1}{c}{$j_2$} & \multicolumn{1}{c}{$j_3$} & \multicolumn{1}{c}{$6 \;\! \hat{f}_{j_1} \hat{f}_{j_2} \hat{f}_{j_3}$} & \multicolumn{1}{c}{$6 \;\! \hat{f}_{j_1} \hat{f}_{j_2} \hat{f}_{j_3}/\mu_3$} \\
    \hline
    $\set{4,9}$ & $\set{4,11}$ & $\set{9,11}$ & $0.0400$ & $0.4910$\\
    $\set{4}$ & $\set{9}$ & $\set{4,9}$ & $0.0137$ & $0.1681$\\
    $\set{4,9}$ & $\set{12}$ & $\set{4,9,12}$ & $0.0021$ & $0.0255$\\
    $\set{9,11}$ & $\set{12}$ & $\set{9,11,12}$ & $0.0015$ & $0.0185$\\
    $\set{4,11}$ & $\set{4,9,12}$ & $\set{9,11,12}$ & $0.0012$ & $0.0146$\\
    \hline
    \multicolumn{2}{c}{} & \multicolumn{1}{r}{Total:} & \multicolumn{1}{c}{$0.0584$} &\multicolumn{1}{c}{$0.7177$} \\
    \end{tabular}\\[10pt]
\end{table}

In summary, by analyzing the trait's third moment in the Fourier domain, we have identified that about $72\%$ of the trait's skew is due to relationships between the genes and their interactions within the tetrahedron in Figure~\ref{Fig:Tetrahedron}. A more thorough analysis of the statistics of this trait and synergies within the gene network can be found in~\cite{herman2026networkSynergiesFourier}.

\subsection{Multidimensional example: Design of a nonlinear process} \label{Sec:Design_DFT_example}
A nonlinear system may have a description which is particularly simple or elegant on the Fourier side. Consider the classic gambling example of summing the payoff from~$n$ fair coin tosses. The sample space for this setup is $G=\Int_2^n$. Section~\ref{Sec:BinomialDist_Stats} discusses how a Bernoulli process, such as this, can be viewed through the Fourier domain. Since each coin flip is independent, the Fourier representation is sparse: a vector of size~$2^n$ where the only nonzero entries are the~$n$ direct effects, so the relative sparsity is $n/2^n$.

For a small, concrete example, let $n=4$ and suppose that heads pays \$$1$ and tails loses \$$1$. This means we set $d=-1$ in~\eqref{Eq:DirectEffect_fhat=d} (if success was tails, then we would have used $d=+1$) so that the Fourier representation is the vector
$$
\hatbsf \,=\, \gr{0,-1,-1,0,-1,0,0,0,-1,0,0,0,0,0,0,0}^\top
$$
with the direct effects at entries $j=1=\vect{\bcode{1},\bcode{0},\bcode{0},\bcode{0}}$, $j=2=\vect{\bcode{0},\bcode{1},\bcode{0},\bcode{0}}$, $j=4=\vect{\bcode{0},\bcode{0},\bcode{1},\bcode{0}}$, and $j=8=\vect{\bcode{0},\bcode{0},\bcode{0},\bcode{1}}$. The inverse Fourier transform~\eqref{Def:inverseDFT} with the Sylvester-Hadamard matrix~$\bs{U}_{\Int_2^4}$ yields the vector of payoffs
\begin{eqnarray*}
\bs{f} &=& \bs{U}_{\Int_2^4}^*\hatbsf \\
&=& \gr{-4, -2, -2, 0, -2, 0, 0, 2, -2, 0, 0, 2, 0, 2, 2, 4}^\top
\end{eqnarray*}
whose histogram is the binomial distribution on the left side of Figure~\ref{Fig:BinaryDist_n4and14}. From Table~\ref{Tab:StandMoments BinomDist} in Section~\ref{Sec:BinomialDist_Stats}, we know that the skewness and kurtosis are $\gamma=0$ and $\kappa = 3-2/4 = 2.5$. Increasing the number of flips to $n = 14$ results in another binomial distribution that is approaching Gaussianity, seen on the right side, with $\gamma=0$ and $\kappa = 3-2/14 = 2.86$.

\begin{figure}[!tb]
\centering
\includegraphics[scale=0.5, trim={20 2 10 0mm},clip] 
{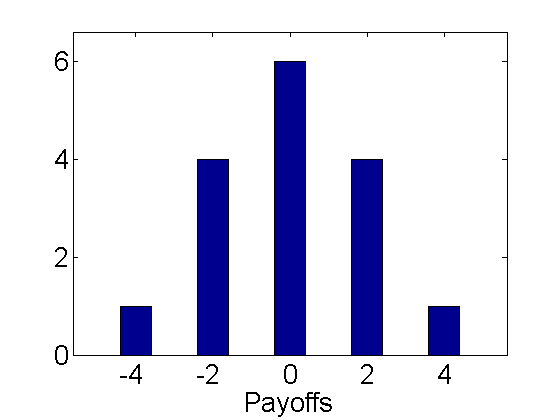}
\quad
\includegraphics[scale=0.5, trim={5 2 10 0mm},clip] 
{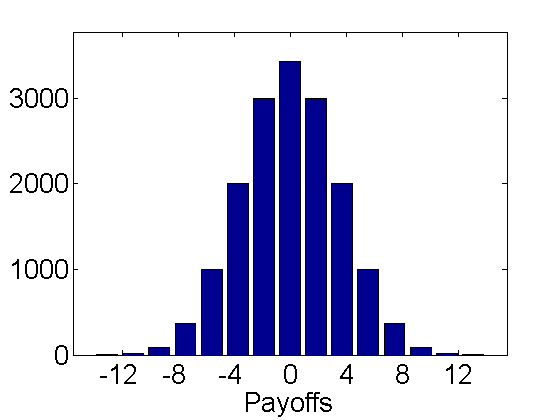}
\vspace{-5pt}
\caption{Binomial distributions from $n$ fair coin flips. \textbf{(Left)} For $n=4$: $\gamma=0$, $\kappa=2.5$; \textbf{(Right)} For $n=14$: $\gamma=0$, $\kappa=2.86$. \label{Fig:BinaryDist_n4and14}}
\end{figure}

The game may be enlivened by small side bets, which can be represented by Fourier entries of $\deg$-$2$ and higher. These side bets (i.e., interactions) introduce nonlinearities in the payoff computations and distort the distribution of payoffs. A simple example is to bet on every combination of pairs of coin flips, losing an amount~$a$ for any two that match, and gaining~$a$ for any two that do not match. This is encoded by setting $\hat{f}_j=-a$ at every index of $\deg$-$2$: $j=3=\vect{\bcode{1},\bcode{1},\bcode{0},\bcode{0}}$, $j=5=\vect{\bcode{1},\bcode{0},\bcode{1},\bcode{0}}$, $j=6=\vect{\bcode{0},\bcode{1},\bcode{1},\bcode{0}}$, $j=9=\vect{\bcode{1},\bcode{0},\bcode{0},\bcode{1}}$, $j=10=\vect{\bcode{0},\bcode{1},\bcode{0},\bcode{1}}$, and $j=12=\vect{\bcode{0},\bcode{0},\bcode{1},\bcode{1}}$, so now
$$
\hatbsf \,=\, \gr{0,-1,-1,-a,-1,-a,-a,0,-1,-a,-a,0,-a,0,0,0}^\top.
$$
Set~$a$ at $10$ cents, for instance. The vector of payoffs from the inverse Fourier transform now becomes
$$
\bs{f} \,=\, \gr{-4.6, -2.0, -2.0, \:\!0.2, -2.0, \:\!0.2, \:\!0.2, \:\!2.0, -2.0, \:\!0.2, \:\!0.2, \:\!2.0, \:\!0.2, \:\!2.0, \:\!2.0, 3.4}^\top
$$
whose histogram is shown on the left side of Figure~\ref{Fig:BinaryDist_n4and14_a10cents}. Compared with the binomial distribution in Figure~\ref{Fig:BinaryDist_n4and14} the nonlinear interactions have caused skewness of $\gamma = -0.44$ to appear and kurtosis to increase to $\kappa = 2.69$. As the number of coin tosses increases, so does the toll extracted by the side bets. At $n = 14$, still with $a = 0.1$, we see on the right side of Figure~\ref{Fig:BinaryDist_n4and14_a10cents} marked skewness of $\gamma = -0.99$ and kurtosis of $\kappa = 4.09$.

\begin{figure}[!t]
\centering
\includegraphics[scale=0.5, trim={20 1 10 0mm},clip] 
{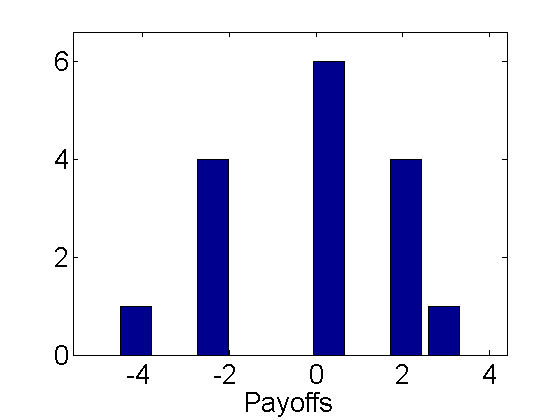}
\quad
\includegraphics[scale=0.5, trim={5 1 10 0mm},clip] 
{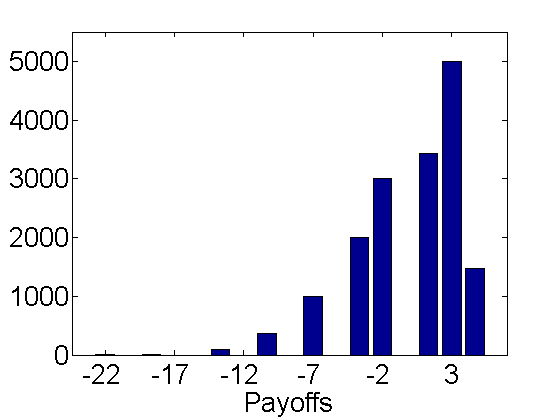}
\vspace{-5pt}
\caption{The histograms resulting from a small side bet of $a=0.1$. Nonlinear interactions cause distortions from the binomial distributions in Fig.~\ref{Fig:BinaryDist_n4and14}. \textbf{(Left)} For $n=4$ coin flips: $\gamma = -0.44$, $\kappa = 2.69$; \textbf{(Right)} For $n=14$ coin flips: $\gamma = -0.99$, $\kappa = 4.09$. \label{Fig:BinaryDist_n4and14_a10cents}}
\end{figure}

This example may be generalized. We may take the index loci as vertices of a graph and include an entry~$-a$ (with $a>0$) for a graph edge, thus creating a variety of nonlinear distributions. Theorem~\ref{Thm:mthMoment_a} permits facile computation of moments of the corresponding distributions, by geometric combinatorics. For example, out of mathematical curiosity, we have computed the moments of the Petersen graph, an unusual list of graph invariants.

Returning to the complete graph introduced above in a gambling context, the second moment for the complete graph on~$n$ vertices, each worth, say,~$d = \pm1$ units, and $n\gr{n-1}/2$ edges each worth~$-a$ units results in
$$
\sigma^2 =\, nd^2 + {n\gr{n-1}} \gr{-a}^2/2 \,=\, n + {n\gr{n-1}} a^2/2.
$$

Now consider the third moment,~$\mu_3$. The index annihilation constraint shows that only two configurations are possible: a ``barbell'' consisting of two vertices connected by an edge, and a triangle consisting of three edges. A barbell occurs for every edge~$uv$ with  endpoints~$u$ and~$v$ (there are $n\gr{n-1}/2$ such), contributing $\gr{-a}d^2 = -a$ for six orderings of the constituents. The overall contribution is $3n\gr{n-1}\gr{-a}$. There are $n\gr{n-1}\gr{n-2}/6$ triangles $uvw$ and the three edges come in six sequences. The overall contribution is $n\gr{n-1}\gr{n-2}\gr{-a}^3$. There are no viable index cancellations of higher order, so
$$
\mu_3 \,=\, -3n\gr{n-1}a - n\gr{n-1}\gr{n-2}a^3.
$$

By similar geometric constructions, one shows
$$
\mu_4 \,=\, n\gr{3n-2} + 3n\gr{5n^2 - 13n + 8}a^2 + n\gr{15n^3 - 78n^2 + 131n - 68}a^4/4.
$$
These formulae are specific versions of~\eqref{Eq:variance_|fj|^2}, \eqref{Eq:mu2_f_Z2n}--\eqref{Eq:mu4_f_Z2n} and demonstrate astonishing condensation of the awkward sums of powers of linear expressions which arise in the direct computation of moments in~\eqref{Def:mthRawCentralMoments} and~\eqref{Def:Variance}. Notice, setting parameter~$a=0$ in these expressions we obtain \eqref{Eq:var_binom}--\eqref{Eq:mu4_binom} with $d=\pm1$; hence increasing~$a$ from zero, they illustrate how including pairwise interactions permits a smooth transition from just linear effects to more realistic nonlinear systems.

These distributions corresponding to complete graphs are prototypes of distributions which deviate from normality because pairwise interactions have been imposed. The effect has been to induce skew and to increase kurtosis (at least initially, as~$a$ increases from zero).

\subsection{Feasibility constraint} \label{Sec:FeasibilityConstraint_DFT_example}
We now illustrate how Theorem~\ref{Thm:mthMoment_a} can be used as a feasibility constraint. Assume function $\bs{f}:G\rightarrow\C$ and $\bs{f} \! \stackrel{\mathcal{F}}\rightleftharpoons \! \hatbsf$ are a Fourier pair, but~$\bs{f}$ is unknown and~$\hatbsf$ is somehow incomplete. If, say,~$\mu_3$ and~$\mu_4$ (or equivalently~$\sigma^2,\gamma,\kappa$) of~$\bs{f}$ are known or assumed, then using~\eqref{ThmEq:mthMoment_mu} we could require
$$
\mu_3 \;= \sum_{\substack{k,l,m\:\!\in\:\!G\backslash0\\[2pt] k \oplus l \oplus m = 0}}
    \hat{f}_k \hat{f}_l \hat{f}_m
\;\;\qquad\text{and}\qquad\;\;
\mu_4 \;= \sum_{\substack{k,l,m,p\:\!\in\:\!G\backslash0\\[2pt] k \oplus l \oplus m \oplus p = 0}}
    \hat{f}_k \hat{f}_l \hat{f}_m \hat{f}_p
$$
in a search or optimization algorithm. The key is:
\begin{enumerate}
\item[(i)] we have closed-form expressions to evaluate higher-order statistics using whatever information may be available, simultaneously, on both sides of the Fourier transform,
\item[(ii)] the Fourier coefficients allowed in each term must satisfy the annihilation condition, thereby limiting the domain of possible solutions.
\end{enumerate}
Note, it can be advantageous to group like terms together (as is done in Section~\ref{Sec:CentralMoments_Z2n}).

\begin{remark} \label{Rem:Normal_constraint}
Typically, higher-order statistics are used when a function's distribution is non-normal. But our tools also apply to normal distributions when used as a constraint. For a normal distribution, the skewness $\gamma_\text{normal} = 0$ and kurtosis $\kappa_\text{normal} = 3$, so the Fourier coefficients would have to satisfy $0 = \sum \hat{f}_k \hat{f}_l \hat{f}_m$ and $3\sigma^4 = \sum \hat{f}_k \hat{f}_l \hat{f}_m \hat{f}_p$.
\end{remark}

\subsubsection{X-ray crystallography application}
Consider an $N_1\text{-by-}N_2$ image~$\bs{f}$ defined on the group $G=\Int_{N_1}\!\times\Int_{N_2}$. Assume the $\abs{G}=N_1N_2$ pixels have been  vectorized to fit our model~\eqref{Def:forwardDFT}; the associated DFT matrix is $\bs{U}_{\:\!\!G} = \bs{U}_{\Int_{N_1}} \!\otimes \bs{U}_{\Int_{N_2}}$. For example, in X-ray crystallography we want to learn an image~$\bs{f}$ from just its Fourier magnitude information, $\abs{\hat{f}_j}$ for all $j\in G$. The goal of recovering the correct phase for each~$\hat{f}_j$ has a long history in inverse problems~\cite{wolf2011historyXrayPhaseProblem}.

Let $\hat{g}_j \,=\, \abs{\hat{f}_j} \:\! \e^{\imag\phi_j}$, where~$\phi_j\in\vect{0,2\pi}$ is a possible phase angle for the \jth\ Fourier coefficient. The candidate Fourier transform~$\hat{\bs{g}} = \Gr{\hat{g}_j}_{j\in G}$ yields an image~$\bs{g} \stackrel{\mathcal{F}}\rightleftharpoons \hat{\bs{g}}$ that satisfies $\abs{\hat{g}_j} = \abs{\hat{f}_j}$ for all $j\in G$.

Even though we do not know image~$\bs{f}$, suppose we have \emph{a priori} information (perhaps from an earlier step in an iterative algorithm) on the third central moment of its pixels,~$\mu_3$. Then in addition to whatever techniques are typically used to recover $\Gr{\phi_j}_{j\in G}$, we can add the constraint
\begin{eqnarray*}
\mu_3 &=&
    \sum_{\substack{k,l,m\:\!\in\:\!G\backslash0\\[2pt] k \oplus l \oplus m = 0}}
   \hat{g}_k \:\! \hat{g}_l \:\! \hat{g}_m \\[5pt]
&=&
    \sum_{\substack{k,l,m\:\!\in\:\!G\backslash0\\[2pt] k \oplus l \oplus m = 0}} \abs{\hat{f}_k} \:\! \abs{\hat{f}_l} \:\! \abs{\hat{f}_m} \, \e^{\imag\gr{\phi_k + \phi_l + \phi_m}}.
\end{eqnarray*}
Crucially, the coefficients allowed in each term must satisfy the {$k \oplus l \oplus m = 0$} condition --- this extra constraint will filter out \emph{unfeasible phase combinations.} It is interesting to reshape vector~$\hat{\bs{g}}$ on an $N_1\text{-by-}N_2$ grid as a $2$-dimensional Fourier transform to see any patterns in the locations of the coefficients in each valid triple, $\hat{g}_k \:\! \hat{g}_l \:\! \hat{g}_m$. If desired, the same process can be carried out for other high-order moments.

\section{Conclusion and future directions} \label{Sec:Conclusion}
Artificial intelligence (AI) is making it increasingly common for us to deal with functions where the relation of input to output is not exactly known. Functions that depend on the inputs of multiple factors, each of which is restricted to a finite handful of alternates, are a salient case. If each such factor can be treated as a cycle of alternates, such a function is quite analogous to a time series, but one composed of multiple individual cycles or factors.

The simplest version of multi-factor functions happens when there are only two alternatives to each factor, the case of Boolean or pseudo-Boolean functions. Thus, in choice theory, each factor can vote \bcode{Yes} or \bcode{No} and the function amalgamates these into an overall decision. In the theory of markets, each participant can be inclined to \bcode{Buy} or \bcode{Sell}, resulting in an overall demand. In genetics, each locus may contribute more or less to a particular phenotypic trait.

The normal distribution can be an admirable device for predicting the net effect of a large number of individual inputs on some variable of interest. This distribution is justified when the function combining these inputs is linear, and it is quite a useful approximation even in the face of significant nonlinearities in the unknown amalgamating function.

Thus, in biological applications, some have treated the genome-phenotype correspondence as largely linear: \emph{``One might say that to first approximation, Biology = linear combinations of nonlinear gadgets, and most of the variation between individuals in a species is due to the (linear) way gadgets are combined, rather than in the realization of different gadgets in different individuals.''}~\cite{raben2022polygenicPrediction}
In contrast, our approach recognizes that even small epistatic interactions may be so numerous as to alter the overall shape of a trait's distribution.

The successes of AI should not lead us to lose all awareness that the function connecting inputs to outcomes is not fully known. Here, the Fourier viewpoint is of great conceptual value, since it clearly distinguishes linearity from a hierarchy of nonlinear effects. \emph{Ab initio,} researchers on complex biological traits have spoken of ``direct gene effects'' or of ``epistatic interactions" without explicit recognition that these can be cast as concepts from the Fourier side.

Fourier methods are well known to help analyze statistical properties of time series, to detect and highlight recurrent features, but the extension to multi-factor functions is less universally adopted. A Fourier approach sheds light on otherwise mysterious phenomena.

An analogue of Theorem~\ref{Thm:mthMoment_a} is possible for multidimensional cumulants, writing each cumulant as a sum of products with annihilating indices. This approach melds well with Gian-Carlo Rota's work~\cite{rota2000combinatoricsCumulants}, treating random variables as umbrae~\cite{rota2001twelveProblemsProbability}. Umbrae are equipped with an \emph{evaluation} operator~$E$, where~$E\gr{\alpha}$ is the expectation of~$\alpha$, when the umbra~$\alpha$ is a random variable. Rota shows the existence of cumulants where the polynomials defining cumulants from moments eliminate products with disjoint supports, ensuring additivity of independent cumulants.

Skewness and leptokurtosis are ubiquitous in empirical data distributions~\cite{taleb2023statistical}. Our Theorem~\ref{Thm:mthMoment_a} may provide a clue to this phenomenon.  In an appropriate setup, a negative interaction may signal a diminishing return relative to the sum of direct effects. When the interaction coefficients are predominantly negative, then right-sided skewness increases, since products of three negative contributions will all be aligned. A similar alignment will tend to increase kurtosis, explaining the many ``fat tails'' being discovered empirically~\cite{doro2026submodularityFourier}.

Persistent deviations from the normal distribution may be sentinels, signaling the presence of hidden nonlinearities in systems which depend on a large set of discrete inputs.

\section{Acknowledgements}
MAH thanks Jassem Shahrani and Zubin Gautam, and SD thanks Stanley Chang for many interesting discussions.

\appendix
\section{Time series \mth-order moment due to $n$ factors} \label{App:TimeSeries_mthOrderMoment_nDim}
Incorporating multidimensional Fourier transforms into the statistical techniques of time series is straightforward, although it demands extra attention to multidimensional indices. As before, let function~$\bs{f} = \gr{f_0,f_1,\ldots,f_{\abs{G}-1}}^\top$ be defined on finite abelian group~$G$~\eqref{Eq:FiniteAbelianGroup_G}, where the \ith\ element, $f_i$, is a function of $n$ factors (each of possibly different size) via the $n$-dimensional index~$i\in G$~\eqref{Eq:i_n-tuple}.

For example, let~$r_m^\bs{f}$ be the \emph{\mth-order moment of~$\bs{f}$ in one dimension} based on $m-1$ time differences or lags, $\set{i_q}_{q=1}^{m-1}$ with $i_q\in\Int_N$; it is defined in~\cite[Eq.~(3.95)]{nikias1993higher} as
\begin{equation} \label{Def:mth-order_moment-timeSeries}
r_m^\bs{f}\gr{i_1,\ldots,i_{m-1}} \,:=\, \frac{1}{\abs{G}} \sum_{i\in G} f_i \:\! f_{i\oplus i_1} \cdots f_{i\oplus i_{m-1}}.
\end{equation}
Here, $G=\Int_N$. Extending the domain to finite abelian~$G$~\eqref{Eq:FiniteAbelianGroup_G}, i.e., from~$1$ to $n$ factors, can be achieved simply by treating the indices as $n$-tuples, where the \qth\ lag is now $i_q = \vect{\bcode{i}_{q,1},\bcode{i}_{q,2},\ldots,\bcode{i}_{q,n}} \in G$. In this way,~\eqref{Def:mth-order_moment-timeSeries} is the \emph{\mth-order moment of~$\bs{f}$ in~$n$ dimensions}.

Function~$\bs{f}$ has a Fourier transform $\hatbsf \,=\, \gr{\hat{f}_0,\hat{f}_1,\hat{f}_2,\ldots,\hat{f}_{\abs{G}-1}}^\top$ defined in~\eqref{Def:forwardDFT}. As such, the \ith\ element lagged by $i_q$ can be expressed as
$$
f_{i\oplus i_q} = \sum_{j_q\in G} {\hat{f}_{j_q} \prod_{\ell=1}^n \omega_{N_\ell}^{\gr{\bcode{i}_\ell \oplus \bcode{i}_{q,\ell}}\bcode{j}_{q,\ell}}}
$$
where we utilized~\eqref{Eq:U_G(i,j)} along with $j_q = \vect{\bcode{j}_{q,1},\bcode{j}_{q,2},\ldots,\bcode{j}_{q,n}} \in G$.
Following the steps in~\cite[\S3.5.2]{nikias1993higher}, substituting this into~\eqref{Def:mth-order_moment-timeSeries} for each $f_{i\oplus i_q}$, $q=1,2,\ldots,m-1$, ultimately yields
\begin{center}
\fbox{
\addtolength{\linewidth}{-17\fboxsep}
 \begin{minipage}{\linewidth}
\begin{equation} \label{Eq:Nikias_mth-order_moment}
r_m^\bs{f}\gr{i_1,\ldots,i_{m-1}} \;=\; \sum_{j_1\in G}\cdots\!\!\sum_{j_{m-1}\in G} \hat{f}_{j_1}\cdots\hat{f}_{j_{m-1}}\hat{f}_{j_m} \cdot \prod_{\ell=1}^n \omega_{N_\ell}^{\oplus_{q=1}^{m-1} \bcode{i}_{q,\ell}\:\!\bcode{j}_{q,\ell}} \quad
\end{equation}
    \vspace{1pt}
 \end{minipage}
}\\[15pt]
\end{center}
where
$$
\hat{f}_{j_m} \,=\, \frac{1}{\abs{G}}\sum_{i\in G} {f_i \:\!\prod_{\ell=1}^n \omega_{N_\ell}^{\ominus\bcode{i}_\ell \bcode{j}_{m,\ell}}}
$$
is element~$j_m = \vect{\bcode{j}_{m,1},\bcode{j}_{m,2},\ldots,\bcode{j}_{m,n}} \in G$ of the Fourier transform of~$\bs{f}$ with
\begin{equation} \label{Eq:j_m_resonance}
\qquad\qquad \bcode{j}_{m,\ell} \,=\, {\ominus_{q=1}^{m-1} \bcode{j}_{q,\ell}}, \qquad \text{for } \ell=1,2,\ldots,n
\end{equation}
i.e., index~$j_m$ is uniquely determined by the other $j_1,j_2,\ldots, j_{m-1}$.

Eq.~\eqref{Eq:Nikias_mth-order_moment} is essentially the $n$-dimensional version of~\cite[Eq.~(3.117)]{nikias1993higher}. Due to~\eqref{Eq:j_m_resonance}, the sum of the~$m$ frequency indices $\set{j_q}_{q=1}^m$ obeys the closure property
\begin{equation} \label{Eq:Resonance_condition}
j_1 \oplus j_2 \oplus \cdots \oplus j_{m-1} \oplus j_m \,=\, 0
\end{equation}
which is often called the ``resonance condition'' in wave phenomena, e.g., see~\cite[Eq.~(2)]{zakharov1975resonantInteractionNonlinearMedia}.

In this paper we are only concerned with \mth-order autocorrelations when all lags $i_q = \vect{\bcode{0},\bcode{0},\ldots,\bcode{0}} = 0$, so~\eqref{Eq:Nikias_mth-order_moment} simplifies to
\begin{equation} \label{Eq:Nikias_mth-order_moment-zeroLags}
r_m^\bs{f}\gr{0,\ldots,0} \,=\, \sum_{j_1\in G}\cdots\!\!\sum_{j_{m-1}\in G}
\hat{f}_{j_1}\cdots\hat{f}_{j_{m-1}}\hat{f}_{j_m}.
\end{equation}
This, taken with~\eqref{Eq:Resonance_condition}, is the special case of the \mth\ raw moment, $\mu'_m\gr{\bs{f}}$, in our Theorem~\ref{Thm:mthMoment_a}. However, our results also apply to complex-valued~$\bs{f}$ centered about an arbitrary point $a\in\C$. Moreover, for each $m$-coefficient term, $\hat{f}_{j_1}\cdots\hat{f}_{j_{m-1}}\hat{f}_{j_m}$, (there are~$\abs{G}^{m-1}$ such) the $n$-tuple index~$j_m$ in~\eqref{Eq:j_m_resonance} needs to be explicitly computed; yet our approach using autoconvolutions of~$\hatbsf$ in~\eqref{Eq:mu_m^a=g*gtg*g} (via powers of circulant matrices~\eqref{Eq:C_*^mf=C_f^m}) automatically generates all of the $m$-coefficient terms in~\eqref{Eq:Nikias_mth-order_moment-zeroLags} that satisfy~\eqref{Eq:Resonance_condition} without using any combinatorics (see Section~\ref{Sec:ComputationalComplexity}).

\bibliography{ms}
\bibliographystyle{IEEEtran}

\end{document}